\documentclass[a4paper,11pt,reqno]{amsart}
\usepackage{amsmath}
\usepackage[T1]{fontenc}
\usepackage{amssymb}
\usepackage{amsthm}
\usepackage{color}
\usepackage[lmargin=2.5 cm,rmargin=2.5 cm,tmargin=3.5cm,bmargin=2.5cm,
paper=a4paper]{geometry}

\newcommand{\Ab}{\mathbf A}
\newcommand{\Fb}{\mathbf F}

\newcommand{\ab}{\mathbf a}
\newcommand{\R}{\mathbb R}
\newcommand{\C}{\mathbb C}
\newcommand{\Z}{\mathbb Z}
\newcommand{\N}{\mathbb N}

\DeclareMathOperator{\E0}{E_{\rm g.st}}

\DeclareMathOperator{\IM}{Im}
\DeclareMathOperator{\curl}{curl}\DeclareMathOperator{\Div}{div}
 \DeclareMathOperator{\dist}{dist}
\DeclareMathOperator{\supp}{supp} 

\newtheorem{thm}{Theorem}[section]
\newtheorem{prop}[thm]{Proposition}
\newtheorem{lem}[thm]{Lemma}

\theoremstyle{remark}
\newtheorem{rem}[thm]{Remark}

\numberwithin{equation}{section}

\title[3D Ginzburg-Landau  functional]{The ground state energy of the three dimensional Ginzburg-Landau functional\\
{\small
Part~I: Bulk regime}}

\author[S. Fournais]{S\o ren Fournais}
\address[S. Fournais and A. Kachmar$^{a, b}$]
{Department of Mathematical Sciences, University
  of Aarhus, Ny Munkegade, Building
  1530, DK-8000 \AA rhus C, Denmark}
\address{$^a$Lebanese University, Department of Mathematics, Hadath, Lebanon.}
\address{$^b$Lebanese International University, School of arts and sciences, Beirut, Lebanon.}
\email[S. Fournais]{fournais@imf.au.dk} \email[A.
Kachmar]{akachmar@imf.au.dk}

\author[A. Kachmar]{Ayman Kachmar}

\date{\today}
\keywords{Ginzburg-Landau functional, thermodynamic limits, elliptic
  estimates, variational methods, magnetic Schr\"odinger operators,
  semiclassical analysis}

\begin{document}
\begin{abstract}
We consider the Ginzburg-Landau functional defined over a bounded
and smooth three dimensional  domain. Supposing that the magnetic
field is comparable with the second critical field and that the
Ginzburg-Landau parameter is large, we determine a sharp asymptotic
estimate of the minimizing energy. In particular, this shows how
bulk superconductivity decreases in average as the applied magnetic
field approaches the second critical field from below. Other
estimates are also obtained which allow us to obtain, in a
subsequent paper \cite{FKP}, a fine characterization of the second
critical field. The approach relies on a careful analysis of several
limiting energies, which is of independent interest.
\end{abstract}
\maketitle

\section{Introduction and main results}
\subsection{Introduction}
Since the publication of the pioneering books \cite{BBH, JT}, the
analysis of the Ginzburg-Landau functional of superconductivity
became the subject of a vast mathematical literature. In particular,
in the presence of an applied magnetic field and in an asymptotic
limit where the Ginzburg-Landau parameter tends to $\infty$,
different advanced mathematical tools indicates three critical
values of the applied magnetic field, usually denoted by $H_{C_1}$,
$H_{C_2}$ and $H_{C_3}$. These critical fields may be described in a
rough way as follows (see e.g. the book of de\,Gennes \cite{deG}).
If a superconducting sample is subject to a constant applied
magnetic field of intensity $H$, then as long as $H<H_{C_1}$, the
sample is in a pure superconducting state and repels the magnetic
field. If the field is slightly increased above $H_{C_1}$, the
sample looses superconductivity in point defects called vortices,
whose number increases as long as the field $H$ is increased all the
way up to the critical value $H_{C_2}$; this phase of the
superconducting sample is the {\it mixed phase}. Increasing the
field above $H_{C_2}$ destroys superconductivity in the bulk of the
sample, but as long as the field $H$ is below $H_{C_3}$, the sample
carries superconductivity along the surface. Finally, if the field
is increased past $H_{C_3}$, superconductivity is lost in the sample
which switches to the normal state. However, to establish these
results in a rigorous mathematical framework, a variety of notations
and methods are required. We mention here that the analysis of the
presence of vortices is the subject of the monograph \cite{SS} and
references therein, while that of surface superconductivity is that
of \cite{FH-b} (and references therein). The methods originally
developed to the analysis of the Ginzburg-Landau functional are also
appearing relevant in the analysis of other mathematical models of
condensed matter physics like superfluidity, Bose-Einstein
condensates and liquid crystals (see \cite{A, Retal, Pa1} and the
references therein).

In this paper, we focus on the analysis of the Ginzburg-Landau
functional in a three dimensional domain and as the Ginzburg-Landau
parameter tends to $\infty$. The regime of the applied magnetic
field considered here is when the field intensity varies close to
and below the second critical field $H_{C_2}$. In two other
subsequent papers \cite{FKP, Ka}, the analysis is completed to cover
variations of the magnetic field intensity in the regime  $\gg
H_{C_1}$ and increasing all the way up to $H_{C_3}$. The novel issue
in the  results is that previous ones where only present for the two
dimensional Ginzburg-Landau functional, see \cite{FH-b,SS}.

Many questions which are answered successfully for the
two-dimensional Ginzburg-Landau functional remain open  for the
three-dimensional functional.  This includes, among other things, a
determination of the first critical field $H_{C_1}$ as precise as
the one present in two-dimensions\,\footnote{It is only recently
that results related to this regime of the magnetic field and valid
for general domains appeared in \cite{Jetal}. Earlier contributions
include a candidate for the expression of $H_{C_1}$ in the case of
the ball \cite{ABM}, and an expression of $H_{C_1}$ in the case of
`thin' shell domains \cite{CS}.}. Among the other open questions, we
mention the analysis of the mixed phase  and the analysis of the
critical field $H_{C_2}$ (that is part of the subject of this and
the subsequent papers \cite{FKP, Ka}).

However, in three-dimensions, the analysis of the third critical
field $H_{C_3}$ already started in \cite{LP}, then a sharp
characterization of it is given in \cite{FH3d}. The regime of an
external magnetic field strength close to and below $H_{C_3}$ is
discussed in the papers \cite{Al, Pa}. In particular, it is proved
that superconductivity is confined to the surface. These results
will be significantly improved in the subsequent paper \cite{FKP}.

In \cite{Al}, interesting estimates are proved in the regime of
applied fields close to and below $H_{C_2}$, which also is the
regime we treat in this paper. We improve  the estimates in
\cite{Al} by proving sharp estimates. The approach we follow is
considerably different from \cite{Al} and mainly   variational in
nature.

\subsection{The functional and main results}

We consider a bounded and open set $\Omega\subset\R^3$ with smooth boundary. We
 suppose that $\Omega$ models a superconducting sample subject to an applied external magnetic field.
The energy of
the sample is
given by the Ginzburg-Landau functional,
\begin{multline}\label{eq-3D-GLf}
\mathcal E^{\rm 3D}(\psi,\Ab)=\mathcal
E_{\kappa,H}^{\rm 3D}(\psi,\Ab)=
\int_\Omega\left[
|(\nabla-i\kappa
H\Ab)\psi|^2-\kappa^2|\psi|^2+\frac{\kappa^2}{2}|\psi|^4\right]\,dx\\
+\kappa^2H^2\int_{\R^3}|\curl\Ab-\beta|^2\,dx\,.
\end{multline}
Here $\kappa$ and $H$ are two positive parameters; $\kappa$ (the Ginzburg-Landau constant) is a material parameter and $H$ measures the intensity of the applied magnetic field.
The wave function
(order parameter)
$\psi\in H^1(\Omega;\C)$ describes the superconducting properties of the material.
The induced magnetic field is $\kappa H \curl \Ab$, where the potential $\Ab\in
\dot{H}^1_{\Div,\Fb}(\R^3)$, with $\dot H^1_{\Div,\Fb}(\R^3)$ being the natural variational space for the problem defined in
\eqref{eq-3D-hs} below. Finally, $\beta$ is the profile and direction of the external magnetic
field that we choose constant, $\beta=(0,0,1)$.

The space $\dot H^1(\R^3)$ is the homogeneous Sobolev space, i.e. the closure
of $C_c^\infty(\R^3)$ under the norm $u\mapsto\|u\|_{\dot
  H^1(\R^3)}:=\|\nabla u\|_{L^2(\R^3)}$. Let further
  $\Fb(x)=(-x_2/2,x_1/2,0)$. Clearly $\Div \Fb=0$.

We define the space,
\begin{equation}\label{eq-3D-hs}
\dot H^1_{\Div,\Fb}(\R^3)=\{\Ab~:~\Div \Ab=0\,,\quad~{\rm and}\quad
\Ab-\Fb\in \dot H^1(\R^3)\}\,.
\end{equation}

Critical points $(\psi,\Ab)\in H^1(\Omega;\C)\times \dot
H^1_{\Div,\Fb}(\R^3)$ of $\mathcal E ^{\rm 3D}$ are weak solutions of
the Ginzburg-Landau equations,
\begin{equation}\label{eq-3D-GLeq}
\left\{
\begin{array}{lll}
-(\nabla-i\kappa H\Ab)^2\psi=\kappa^2(1-|\psi|^2)\psi&{\rm in}& \Omega
\\
\curl^2\Ab=-\displaystyle\frac1{\kappa H}\IM(\overline{\psi}\,(\nabla-i\kappa
H\Ab)\psi)\mathbf 1_\Omega&{\rm in}& \R^3\\
\nu\cdot(\nabla-i\kappa H\Ab)\psi=0&{\rm on}&\partial\Omega\,,
\end{array}\right.\end{equation}
where $\mathbf 1_\Omega$ is the characteristic function of the domain $\Omega$, and $\nu$ is the unit interior normal vector  of $\partial\Omega$.

Our main results include a sharp asymptotic estimate of the ground state energy which is defined as follows,
\begin{equation}\label{eq-3D-gs}
\E0(\kappa,H)=\inf\big\{
\mathcal E^{\rm 3D}(\psi,\Ab)~:~(\psi,\Ab)\in H^1(\Omega;\C)\times \dot H^1_{\Div,\Fb}(\R^3)\big\}\,.
\end{equation}
The asymptotic estimate of $\E0(\kappa,H)$ involves an auxiliary function $g:[0,\infty)\to[-\frac12,0]$ introduced in \eqref{eq-g(b)} below. The function $g$ is increasing, continuous
and $g(b)=0$ for all $b\geq 1$.

\begin{thm}\label{thm-3D-main}
Let $0<\Lambda_{\rm min}< \Lambda_{\rm max}$. There exist positive constants $C$ and $\kappa_0$ such that if
\begin{align*}
\kappa_0 \leq \kappa,\qquad \Lambda_{\rm min}\leq \frac{H}{\kappa}\leq  \Lambda_{\rm max}\,,
\end{align*}
then the ground state energy in \eqref{eq-3D-gs} satisfies,
\begin{equation}\label{eq-3D-thm}
\Big| \E0(\kappa,H)-g(H/\kappa)|\Omega|\kappa^2\Big|\leq C\kappa^{3/2}\,.
\end{equation}
\end{thm}

Theorem~\ref{thm-3D-main} is an extension to three-dimensional
domains  of the analogous Theorem~1.4 in \cite{SS02} established for
two-dimensional domains. Let us emphasize that the limiting function
$g(\cdot)$ in Theorem~\ref{thm-3D-main} is the same as that in
\cite{SS02}, up to an additive constant.

In the particular regime when $H=\kappa+o(\kappa)$, Theorem~\ref{thm-3D-main} tells
us that $\E0(\kappa,H)=o(\kappa^2)$ as $\kappa\to\infty$. In this case,
the leading order term in the ground
state energy is given in the subsequent part of this paper \cite{FKP}.
Actually, we will prove that there exist two constants $E_{\rm surf}<0$  and $E_2<0$ such that,
\begin{equation}\label{eq-FKP}
\E0(\kappa,H)=E_{\rm surf}\kappa+E_2|\Omega|[\kappa-H]_+^2+o
\bigg(\max\left(\kappa,\left[\kappa-H\right]_+^2\right)\bigg)\,.
\end{equation}
Here we use the function $[x]_+=\max(x,0)$, $x\in\R$. We stress that
the constant $E_{\rm surf}$ depends (only) on the boundary of
$\Omega$, while the constant $E_2$ is universal.
Part of the necessary estimates to prove \eqref{eq-FKP} are given in the last section of this paper.

Furthermore, in the subsequent part of this paper \cite{FKP}, we determine leading order estimates
of the ground state energy when the magnetic field increases all the way up to the critical value $H_{C_3}$. This gives an overall understanding of the ground state energy of three dimensional superconductors subject to strong magnetic fields, with a precise description of the transition from {\it bulk} to {\it surface} regime as well.

The next theorem concerns the behavior of order parameters for critical configurations.
We obtain a sharp asymptotic estimate on the average of the order parameter on small cubes.

\begin{thm}\label{thm-3D-op}
Suppose the assumptions of Theorem~\ref{thm-3D-main} are satisfied. There exist positive constants
$C$, $R_0$  and $\kappa_0$ such that, if
$$\kappa_0\leq\kappa\,,\quad R_0\kappa^{-1}\leq \ell\leq\frac12\,,$$
and
$Q_{\ell}$ is a cube of side-length $\ell$ and parallel to the external magnetic field $\beta$, satisfying $Q_{\ell}\subset\Omega$, then the following is true.
\begin{enumerate}
\item If $(\psi,\Ab)\in H^1(\Omega;\C)\times \dot H^1_{\Div,\Fb}(\R^3)$ is a solution of \eqref{eq-3D-GLeq}, then,
\begin{equation}\label{eq-3D-op'}
\frac1{|Q_{\ell}|}\int_{Q_{\ell}}|\psi|^4\,dx\leq -2g(H/\kappa)
+C\left(\kappa^{-1/2}+\frac1{\ell\kappa}\right)\,,\quad{\rm as}\quad\kappa\to\infty\,.
\end{equation}
\item
If $(\psi,\Ab)\in H^1(\Omega;\C)\times \dot H^1_{\Div,\Fb}(\R^3)$ is a minimizer of \eqref{eq-3D-GLf}, then,
\begin{equation}\label{eq-3D-op}
\left|\frac1{|Q_{\ell}|}\int_{Q_{\ell}}|\psi|^4\,dx+2g(H/\kappa)\right|\leq
C\left(\kappa^{-1/2}+\frac1{\ell\kappa}\right)\,,\quad{\rm as}\quad\kappa\to\infty\,.
\end{equation}
\end{enumerate}
\end{thm}

\begin{rem}
Theorem~\ref{thm-3D-op} still holds true (with a different error term) if cubes are replaced by more general domains that can be approximated by cubes of side-length $\ell$ and $\kappa^{-1}\ll\ell\ll1$ as $\kappa\to\infty$. In particular, it is true for balls of radius $\ell$.
\end{rem}

\begin{rem}\label{rem-1}
The function $g$ is a continuous increasing function, and $g(b)=0$ for all $b\geq 1$ (see Section~\ref{sec-LE2D}). Furthermore, there exists a universal constant $\alpha\in(0,\frac12)$ such that,
$$\forall~b\in[0,1]\,,\quad  \alpha(b-1)^2\leq |g(b)|\leq \frac12(b-1)^2\,.$$
Theorem~\ref{thm-3D-op} then shows that as the applied magnetic field approaches $H_{C_2}$ (i.e. $H\sim b\kappa$ and $b\to1_-$), superconductivity decreases in the bulk like $\sqrt{1-b}$.
\end{rem}

\begin{rem}
In connection with the result of Theorem~\ref{thm-3D-op},
Almog \cite{Al} proves that if the magnetic field is such that $H=b\kappa+o(\kappa)$ with $b\leq 1$, and if  $(\psi,\Ab)$ is a  solution of the Ginzburg-Landau equations \eqref{eq-3D-GLeq},  then
\begin{equation}\label{eq-al}
\frac1{|\Omega|}\int_\Omega|\psi|^4\,dx\leq \frac{C}{\kappa}+2\left[1-\frac{H}{\kappa}\right]_+^2\,.\end{equation}
The estimate \eqref{eq-3D-op'} we give in Theorem~\ref{thm-3D-op} is an improvement of \eqref{eq-al} in the sense that
 it estimates the average of $|\psi|^4$ on small cubes and to leading order.

In the case $b=1$, the estimate of the ground state energy in \eqref{eq-FKP} yields an estimate stronger than the ones in \eqref{eq-al} and Theorem~\ref{thm-3D-op}. We refer to \cite{FKP} for the precise statement.
\end{rem}

\begin{rem}
In light of the result of Theorem~\ref{thm-3D-op}, it is natural to expect that for any solution
 $(\psi,\Ab)$ of the Ginzburg-Landau equation \eqref{eq-3D-GLeq}, the order parameter
$\psi$ satisfies
$$\|\psi\|_{L^\infty(\omega_\kappa)}\leq C\max\left(\left|\frac{\kappa}{H}-1\right|^{1/2},\kappa^{-\delta}\right)\,,$$
 where $C$ is a constant, $\delta\in(0,1)$  and $\omega_\kappa=\{x\in\Omega~:~{\rm dist}(x,\partial\Omega)\geq \kappa^{-\delta}\}$.
Bounds of this type were obtained for the two dimensional equations in \cite{FH-jems, FK, FK1}, but we are not at present able to obtain them in 3D.
\end{rem}

The general technique used in this paper is mainly based on the methods of \cite{FH-b}. After determining {\it a priori} estimates of the solutions, we reduce the problem, via localization techniques,   to that of a three dimensional  model
problem. The latter  is linked to a  two-dimensional problem, which has been studied in \cite{AS, SS02}.

The main concern in \cite{AS, SS02} was the leading order term in the ground state energy for the two-dimensional model problem, without attention to estimates of the remainder terms. However,
precise estimates of those remainder terms are necessary in order to  obtain the correct leading order term
in the energy \eqref{eq-3D-gs}.

We perform a very careful analysis of the two dimensional limiting problem thereby obtaining estimates that, on the one hand, are stronger than those appearing in \cite{AS, SS02}, and on the other hand, when inserted in the original functional in \eqref{eq-3D-GLf}, their contribution to the remainder  terms is negligeable compared with the expected  leading order term in the energy.  On a technical level, a key tool in the control of remainder terms is an $L^\infty$-bound for solutions of a reduced Ginzburg-Landau equation in the plane, which was obtained in \cite{FH-jems}.

Let us point out that, using less sophisticated tools than the ones used in \cite{AS, SS02}, we recover in a unified approach
all the results obtained   in \cite{AS} concerning the limiting problem.  It seems that our approach works for  limiting energies  arising in other contexts.
In \cite{FKP}, through the same approach, we  identify the limiting problem of the {\it surface} regime.

The paper is organized as follows. In Section~\ref{sec-LE}, we analyze the limiting energy.
In Section~\ref{sec-apest}, we give
asymptotic {\it a priori} estimates on solutions of the Ginzburg-Landau equation \eqref{eq-3D-GLeq}. In Section~\ref{sec-locest}, we prove local energy estimates in small cubes for minimizers of the energy in \eqref{eq-3D-GLf}.
Section~\ref{sec-proofs} concludes with the proofs of the theorems announced in the introduction. Section~\ref{sec-hc2}
is devoted to additional estimates related to the case when $H$ is close to $\kappa$, which will be used in \cite{FKP}.
\medskip

\paragraph{\it Remark on notation:}
\begin{itemize}
\item Throughout the paper, we write $\mathcal E$ for the functional $\mathcal E^{\rm 3D}$ in \eqref{eq-3D-GLf}.
\item The letter $C$ denotes a positive constant that is independent of the parameters $\kappa$ and $H$, and whose value  may change from line to line.
\item If $a(\kappa)$ and $b(\kappa)$ are two positive functions, we write $a(\kappa)\ll b(\kappa)$ if $a(\kappa)/b(\kappa)\to0$ as $\kappa\to\infty$.
\item If $a(\kappa)$ and $b(\kappa)$ are two functions with $b(\kappa)\not=0$, we write $a(\kappa)\sim b(\kappa)$ if $a(\kappa)/b(\kappa)\to1$ as $\kappa\to\infty$.
\item If $a(\kappa)$ and $b(\kappa)$ are two positive functions, we write $a(\kappa)\approx b(\kappa)$
if there exist positive constants $c_1$, $c_2$ and $\kappa_0$ such that $c_1b(\kappa)\leq a(\kappa)\leq c_2b(\kappa)$ for all $\kappa\geq\kappa_0$.
\item If $x\in\R$, we let $[x]_+=\max(x,0)$.
\end{itemize}

\section{The limiting energy}\label{sec-LE}

This section contains the study of the large area/volune limit of Ginzburg-Landau functionals with constant magnetic field in 2 and 3 dimensions. These models are fundamental for the results of the paper. In the $2$D case we compare different boundary conditions (Dirichlet in Section~\ref{sec-LE2D} and periodic in Section~\ref{sec:periodic}) and introduce the Abrikosov energy corresponding to restricting the functional to periodic functions in the lowest Landau band in Section~\ref{sec:Abrikosov}.
Finally, in Section~\ref{sec-LE3D} we reduce the $3$ dimensional case to the $2$D one.

\subsection{Two-dimensional limiting energy}\label{sec-LE2D}
\subsubsection{Reduced Ginzburg-Landau functional and thermodynamic limit}\

Given a constant $b\geq 0$ and an open set $\mathcal
D\subset \R^2$, we define the following Ginzburg-Landau energy,
\begin{equation}\label{eq-LF-2D}
G_{b,\mathcal D}(u)=\int_{\mathcal D}\left(b|(\nabla-i\Ab_0)u|^2
-|u|^2+\frac1{2}|u|^4\right)\,dx\,.
\end{equation}
Here $\Ab_0$ is the canonical magnetic potential,
\begin{equation}\label{eq-hc2-mpA0}
\Ab_0(x_1,x_2)=\frac12(-x_2,x_1)\,,\quad\forall~x=(x_1,x_2)\in
\R^2\,.\end{equation}
We will consider the functional $G_{b,\mathcal D}$ first with Dirichlet and later with (magnetic) periodic boundary conditions. It will be clear from the context what is meant.

Consider the functional with Dirichlet boundary conditions and for
$b>0$. If the domain $\mathcal D$ is bounded, completing  the square
in the expression of $G_{b,\mathcal D}$ shows that $G_{b,\mathcal
D}$ is bounded from below. Thus, starting from a minimizing
sequence, it is easy to check that $G_{b,\mathcal D}$ has a
minimizer. A standard application of the maximum principle  shows
that, if $u$ is any minimizer of $G_{b,\mathcal D}$, then
\begin{align}
\label{eq:MaxPr}
|u|\leq 1,\qquad \text{  in } {\mathcal D},
\end{align}
see e.g. \cite{SS}.

Given $R>0$, we denote by  $K_R=(-R/2,R/2)\times(-R/2,R/2)$ a square of side length
$R$. Let,
\begin{equation}\label{eq-m0(b,R)}
m_0(b,R)=\inf_{u\in H^1_0(K_R;\C)}
G_{b,K_R}(u)\,.
\end{equation}

The main concern of this section is the proof of  the following theorem.

\begin{thm}\label{thm-thmd-AS}
Let $m_0(b,R)$ be as defined in \eqref{eq-m0(b,R)}.
\begin{enumerate}
\item For all $b\geq 1$ and $R>0$, we have $m_0(b,R)=0$.
\item If $b\in[0,\infty)$, there exists a constant $g(b)\leq 0$ such that,
\begin{equation}\label{eq-g(b)}
\quad g(b)=\lim_{R\to\infty}\frac{m_0(b,R)}
{|K_R|}\,.
\end{equation}
\item The function $[0,+\infty)\ni b\mapsto g(b)$ is continuous,
non-decreasing, concave and its range is the interval $[-\frac12,0]$.
\item There exists a universal constant $\alpha\in(0,\frac12)$ such
  that,
\begin{equation}\label{eq-prop-g(b)}\forall~b\in[0,1]\,,\quad
  \alpha(b-1)^2\leq|g(b)|\leq \frac12(b-1)^2\,.
\end{equation}
\item There exist universal constants $C$ and $R_0$ such that,
\begin{equation}\label{eq-remainder}
\forall~R\geq R_0\,,~\forall~b\in[0,1]\,,\quad g(b)\leq
\frac{m_0(b,R)}{R^2}\leq g(b)+\frac{C}{R}\,.\end{equation}
\end{enumerate}
\end{thm}

The major part of Theorem~\ref{thm-thmd-AS} is obtained  by Sandier-Serfaty
\cite{SS02} and Aftalion-Serfaty \cite[Lemma~2.4]{AS}. However, the
estimate in \eqref{eq-remainder} is new.  We give an alternative proof
of Theorem~\ref{thm-thmd-AS} through a direct approach\,\footnote{In \cite{AS, SS02}, the authors link the ground state energy $m_0(b,R)$ to that of a full Ginzburg-Landau energy (where the magnetic field is an unknown), then for the latter, they analyze a simultaneous
limit of  both large area and large Ginzburg-Landau parameter.}.

An important key-ingredient in the proof of Theorem~\ref{thm-thmd-AS} is the following abstract lemma.

\begin{lem}\label{gen-lem}
Consider a decreasing function $d:(0,\infty)\to(-\infty,0]$ such that the function\break
$f:(0,\infty)\ni\ell\mapsto \frac{d(\ell)}{\ell^2}\in\R$
is bounded.
\begin{enumerate}
\item
Suppose that there exist constants $C>0$ and $\ell_0>0$ such that  the estimate
\begin{equation}\label{eq-gen-lem-ass}
f(n\ell) \geq f((1+a)\ell)-C\left(a+\frac1{a^2\ell^2}\right)\,,\end{equation}
holds true for all $a\in(0,1)$,
$n\in\mathbb N$ and $\ell\geq\ell_0$.

Then $f(\ell)$ has a limit $A$ as $\ell\to\infty$. Furthermore, for all $\ell\geq 2\ell_0$, the following estimate holds true,
\begin{equation}\label{eq-gen-lem}
f(\ell)\leq A+\frac{2C}{\ell^{2/3}}\,.
\end{equation}
\item
Suppose that there exist constants $C>0$ and $\ell_0>0$ such that  the estimate
\begin{equation}\label{eq-gen-lem-ass'}
f(n\ell) \geq f(\ell)-\frac{C}{\ell}\,,\end{equation}
holds true for all
$n\in\mathbb N$ and $\ell\geq\ell_0$.

Then $f(\ell)$ has a limit $A$ as $\ell\to\infty$. Furthermore, for all $\ell\geq 2\ell_0$, the following estimate holds true,
\begin{equation}\label{eq-gen-lem'}
f(\ell)\leq A+\frac{2C}{\ell}\,.
\end{equation}
\end{enumerate}
\end{lem}
 \begin{proof}
We prove the first conclusion of the lemma. As the proof of the second conclusion is similar, we do not write the details.

We denote by   $A=\displaystyle\limsup_{\ell\to\infty}f(\ell)$. We know that $A$ is finite since the function $f$ is bounded.
Also, $A\leq 0$ since $d$ is non-positive.
We will prove that $\displaystyle\liminf_{\ell\to\infty}f(\ell)\geq A$. This will give us that $f(\ell)$ has limit $A$ as $\ell\to\infty$.

 Suppose that $\varepsilon\in(0,1)$ is a given constant.   We may select $\ell_0\geq1/(\varepsilon^2)$ such that
$f(\ell_0)\geq A-\varepsilon$, and the estimate in \eqref{eq-gen-lem-ass} holds true for all $\ell\geq\ell_0$.

For each $\ell\in(\ell_0/(1+\varepsilon),\infty)$, let $n\in\mathbb N$
be the unique natural number satisfying $n\ell_0\leq(1+\varepsilon)\ell<(n+1)\ell_0$.
By monotonicity of $\ell\mapsto d(\ell)$, we have, $d(\ell)\geq d\big((n+1)\ell_0/(1+\varepsilon)\big)$.
By definition of $f(\ell)$ we get,
$$f(\ell)\geq \left(\frac{(n+1)\ell_0}{(1+\varepsilon)\ell}\right)^2\,f\left(\frac{(n+1)\ell_0}{1+\varepsilon}\right)\,.$$
 Using the bound in \eqref{eq-gen-lem-ass}
with $a=\varepsilon$, we deduce that,
$$f(\ell)\geq
\left(\frac{(n+1)\ell_0}{(1+\varepsilon)\ell}\right)^2\left[f(\ell_0)-C\left(\varepsilon+\frac{2}{\varepsilon^2\ell_0^2}\right)\right]\,.$$
We notice that by our assumption on $n$,
\begin{align*}
\left(\frac{(n+1)\ell_0}{(1+\varepsilon)\ell}\right)^2&\leq \left(\frac{n\ell_0}{(1+\varepsilon)\ell}\right)^2+
\frac{1}{(1+\varepsilon)\ell}\left[
\frac{2n\ell_0}{(1+\varepsilon)\ell}+\frac{\ell_0^2}{(1+\varepsilon)\ell}\right]\\
&\leq1+\frac{2+\ell_0}{\ell}\,.
\end{align*}
Since $d(\ell_0)\leq 0$, we deduce that,
$$f(\ell)\geq \left(1+\frac{2+\ell_0}{\ell}\right)\left[f(\ell_0)-C\left(\varepsilon+\frac1{\varepsilon^2\ell_0^2}\right)\right]\,.$$
Consequently, we get by taking $\liminf$ on both sides above,
$$\liminf_{\ell\to\infty}f(\ell)  \geq A-C(\varepsilon+\varepsilon^2)\,.$$
Taking $\varepsilon\to0_+$, we get that
$$\liminf_{\ell\to\infty}f(\ell)  \geq A\,.$$
The estimate in \eqref{eq-gen-lem-ass} applied with arbitrary $n\in\mathbb N$ and $a\in(0,1)$ yields,
$$
f(nL)\geq f((1+a)L)-C\left(a+\frac1{a^2L^2}\right)\,,
$$
for all $L\geq \ell_0$. Making $n\to\infty$ we get,
$$
A\geq f((1+a)L)-C\left(a+\frac1{a^2L^2}\right)\,.
$$
Selecting $L=\ell/(1+a)$ and $a=\ell^{-2/3}$ we obtain that,
$$A\geq f(\ell)-\frac{2C}{\ell^{2/3}}\,.$$
This finishes the proof of assertion (1) in Lemma~\ref{gen-lem}.
\end{proof}

In the next lemma, we give rough bounds on the energy $m_0(b,R)$.
\begin{lem}\label{lem-rough'}
For all $b\geq0$ and $R>0$,  we have,
$$-\frac{[1-b]_+^2 R^2}2\leq m_0(b,R)\leq 0\,,$$
where $m_0(b,R)$ is defined in \eqref{eq-m0(b,R)}.

Furthermore, there exist  universal constants $\alpha\in(0,\frac12)$ and $C>0$ such that, for all $b\in[0,1)$ and $n\in\N$, we have,
\begin{equation}\label{eq-rough-ub'}
\frac{m_0(b,2\sqrt{2\pi} \,n)}{8\pi n^2}
\leq - \alpha(1-b)^2+\frac{C(1-b)^2}n\,.
\end{equation}
\end{lem}
\begin{proof}
Let $u\in H^1_0(K_R)$. We extend $u$ to
a function $\widetilde u\in H^1(\R^2)$ by setting $\widetilde u=0$ in $\R^2\setminus K_R$. Then,
$$G_{b,K_R}(u)=\int_{\R^2}\left(b|(\nabla-i\Ab_0)\widetilde u|^2-|\widetilde u|^2+\frac12|\widetilde u|^4\right)\,dx\,.$$
Recall that the lowest
lowest eigenvalue of the Schr\"odinger operator with constant magnetic field $-(\nabla-i\Ab_0)^2$ in $L^2(\R^2)$
is equal to $1$. So, we get by the variational principle,
\begin{align*}
G_{b,K_R}(u)\geq \int_{\R^2}\left((b-1)|\widetilde u|^2+\frac12|\widetilde u|^4\right)\,dx
=\int_{K_R} \left((b-1) |u|^2 + \frac{1}{2} |u|^4\right) \,dx\,.\end{align*}
Since $u\in H^1_0(K_R)$ was arbitrary this clearly gives  the lower bound of the lemma when $b\geq1$. When $0<b<1$, we can complete the square to get,
\begin{align*}
\int_{K_R} \left((b-1) |u|^2 + \frac{1}{2} |u|^4\right) \,dx&=\frac12\int_{K_R} \left(\big((b-1) -|u|^2\big)^2
 - (1-b)^2\right) \,dx\\
&\geq -\frac12(1-b)^2R^2\,,
\end{align*}
which thereby finishes the proof of the lower bound of the lemma.

Since $G_{b,K_R}(0)=0$, we get the trivial upper bound $m_0(b,R)\leq 0$. This proves the first estimate in the lemma.

The proof of the estimate in \eqref{eq-rough-ub'} consists of computing the energy of a test function constructed in \cite{SS02}.
A similar bound (more in line with the techniques of the present
paper) can be obtained by using the `lowest Landau band' functions
introduced below. We will briefly sketch such a calculation. With
$R=2\sqrt{2\pi}$ the space $L_R$ defined in
Proposition~\ref{prop-hc2-poperator} is non-empty and contains a
non-zero magnetic periodic function $v\in E_{R=2\sqrt{2\pi}}$\,,
where the space $E_{R}$ is defined in \eqref{eq-hc2-space1}. By
defining $u=\lambda\sqrt{1-b}\,v/\|v\|_{L^2(K_R)}$, we may select
$\lambda>0$ sufficiently small and a constant $\alpha\in(0,1/2)$
such that,
\begin{align} G_{b,K_R}(u) = -\alpha (1-b)^2 <0.
\end{align}
We can take the magnetic periodic function $u$ and cut it down to a
box of size $nR$ (in order to satisfy the Dirchlet boundary)
condition.
Upon calculating the energy on $K_{nR}$ of this function
we get the bound  \eqref{eq-rough-ub'}, where the last term comes
from the localization error.

This finishes the proof of the lemma.
\end{proof}

The next lemma establishes monotonicity properties of $m_0(b,R)$ with respect to $R$.

\begin{lem}\label{lem-monton'}
For all $b>0$, the function $R\mapsto m_0(b,R)$ defined in \eqref{eq-m0(b,R)} is monotone non-increasing.

Furthermore, there exist universal constants $C>0$ and $\ell_0>0$ such that, for all $b>0$  and $\ell\geq\ell_0$, we have the estimate,
$$\frac{m_0(b,nR)}{ (nR)^2}\geq \frac{m_0(b,R)}{R^2}-\frac{C(b+1)}{R}\,.$$
\end{lem}
\begin{proof}
We start by proving that the function $R\mapsto m_0(b,R)$ is monotone. Let $r>0$ and $u\in H^1_0(K_R)$. We extend $u$ to a function $\widetilde u\in H^1_0(K_{R+r})$ by setting $\widetilde u=0$ outside $K_R$.
In this way, we get,
$$G_{b,K_{R}}(u)=G_{b,K_{R+r}}(\widetilde u)\geq m_0(b,R+r)\,.$$
Since $u\in H^1_0(K_R)$ was arbitrary
this proves the montonicity of $m_0(b,R)$ with respect to $R$.

We prove the lower bound of the lemma.
If  $j=(j_1,j_2)\in\Z^2$, we denote by
$$K_{j}=I_{j_1}\times I_{j_2}\,,$$
where
$$\forall~m\in\Z\,,\quad I_m=\bigg{(}\frac{2m+1-n}2-\frac{1}2\,,\,\frac{2m+1-n}2+\frac{1}2\bigg{)}\,.$$
For all $R>0$, we set
$$\mathcal  K_{R,j}=\{R x~:~x\in K_{j}\}\,.$$
Let $\mathcal J=\{j=(j_1,j_2)\in \Z^2~:~0\leq  j_1,j_2\leq n-1\}$
and $K_{nR}=(-nR,nR)\times(-nR,nR)$. Then the family
$(\overline{\mathcal K_{R,j}})_{j\in\mathcal J}$ is a covering of
$K_{nR}$, and is formed exactly of $n^2$ squares. Let $u_{nR}$ be a
minimizer of $G_{b,K_{nR}}$, i.e. $G_{b,K_{nR}}(u_{nR})=m_0(b,nR)$.
We have the obvious decomposition,
\begin{equation}\label{eq-IMS'}
\int_{K_{nR}}|u_{nR}|^4\,dx= \sum_{j\in\mathcal J}
\int_{\mathcal K_{R,j}}|u_{nR}|^4\,dx.
\end{equation}
Let $\chi_{R,j}\in C_c^\infty(\R^2)$ be a function satisfying,
$$0\leq\chi_{R,j}\leq 1\,,\quad {\rm in~}\R^2\,,\quad
\supp\chi_{R,j}\subset \mathcal K_{R,j}\,,\quad \chi_{R,j}=1\quad{\rm in}~K_{R-1,j}\,,\quad |\nabla\chi_{R,j}|\leq C\quad{\rm in}~\R^2\,,$$
for some universal constant $C$.

Since $u_{nR}$ satisfies the equation $-b(\nabla-i\Ab_0)^2u_{nR}=(1-|u_{nR}|^2)u_{nR}$, it results from an integration by parts that $G_{b,K_{nR}}(u_{nR})=-\displaystyle\frac12\int_{K_{nR}}|u_{nR}|^4\,dx$ and,
$$\int_{K_{nR}}\left(b|(\nabla-i\Ab_0)\chi_{R,j}u_{nR}|^2-|\chi_{R,j}u_{nR}|^2+\chi_{R,j}^2|u_{nR}|^4\right)\,dx=
b\int_{K_{nR}}\big||\nabla \chi_{R,j}|u_{nR}\big|^2\,dx\,.
$$
Using that $|u_{nR}|\leq 1$, it is easy to check that (with a new constant $C$),
\begin{align}\label{eq:L-4Energy}
G_{b,K_{nR}}(\chi_{R,j}u_{nR})\leq -\frac12\int_{\mathcal K_{R,j}}|u_{nR}|^4\,dx+C(b+1)R\,.
\end{align}
Since each $\chi_{R,j}u_{nR}$ has support in a square of side length $R$, we get by magnetic translation invariance,
$$G_{b,K_{nR}}(\chi_{R,j}u_{nR})\geq m_0(b,R)\,.$$
We insert this lower bound  into \eqref{eq:L-4Energy}, then we take the sum over all $j\in\mathcal J$. In this way we get,
\begin{align*}
n^2\, m_0(b,R)&\leq -\frac12\sum_{j\in\mathcal J}\int_{\mathcal K_{R,j}}|u_{n,R}|^4\,dx+Cn^2(b+1)R\\
&=-\frac12\int_{K_{nR}}|u_{nR}|^4\,dx+Cn^2(b+1)R\,.\end{align*}
To finish the proof of Lemma~\ref{lem-monton'}, we just recall that,
$$m_0(b,nR)=G_{b,K_{nR}}(u_{nR})=-\frac12\int_{K_{nR}}|u_{nR}|^4\,dx\,.$$
\end{proof}

The next lemma treats the specific case $b=0$.

\begin{lem}\label{lem-b=0}
If $b=0$ and  $R>0$, then
$m_0(b,R)=-\displaystyle\frac{R^2}2$.
\end{lem}
\begin{proof}
Let $u\in H^1_0(K_R)$. It is easy to see, by completing the square, that if $b=0$,
\begin{align}
G_{b,K_R}(u) = \int_{K_R} \frac{1}{2} \big( |u|^2-1\big)^2\,dx -\frac{R^2}{2}.
\end{align}
Consequently, we get $m_0(b,R)\geq-R^2/2$ when $b=0$. We will prove that $m_0(b,R)\leq-R^2/2$.
If $n>R$ is a natural number, let $u_n \in C_c^{\infty}(K_R)$ be a function satisfying,
\begin{align*}
0\leq u_n \leq 1, \qquad u_n \equiv 1 \text{ on } K_{R-\frac1n}.
\end{align*}
Then $m_0(b,R)\leq G_{b,K_R}(u_n)$ for all $n>R$.
It is easy to check that Lebesgue's dominated convergence theorem gives,
$\displaystyle\lim_{n\to\infty} G_{b,K_R}(u_n)=-\frac{R^2}2$, thereby proving that $m_0(b=0,R)\leq-R^2/2$.
\end{proof}

\begin{proof}[Proof of Theorem~\ref{thm-thmd-AS}]
The proof is separated into several steps.

{\it Step~1.}
Notice that the first conclusion in the theorem is a direct consequence of Lemma~\ref{lem-rough'}.
We prove that $\displaystyle\frac{m_0(b,R)}{R^2}$ has a limit $g(b)$ as $R\to\infty$. Thanks to Lemma~\ref{lem-b=0}, we have nothing to prove when $b=0$. Also, Lemma~\ref{lem-b=0} gives us that $g(0)=-1/2$.

Suppose that $0<b\leq 1$. Let
$$f_b(R)=\frac{m_0(b,R)}{R^2}\,,\quad d_b(R)=m_0(b,R)\,.$$
Thanks to Lemmas~\ref{lem-rough'} and \ref{lem-monton'}, the functions $f_b$ and $d_b$ satisfy the assumptions of
Lemma~\ref{gen-lem}. Therefore, we conclude that $f_b(R)$ has a limit $g(b)$ as $R\to\infty$. Furthermore, since $b\in[0,1]$, we infer from the estimate of Lemma~\ref{gen-lem},
$$f_b(R)\leq g(b)+\frac{C}{R}\,,$$
where $C$ is a universal constant. This proves the upper bound in assertion (5) of Theorem~\ref{thm-thmd-AS}.

{\it Step~2.}
In this step,
we prove the lower bound $m_0(b,R)\geq g(b) R^2$. Let $n\in\N$ and $R>0$. We use the notation
$K_R=(-\frac{R}2,\frac{R}2)\times(-\frac{R}2,\frac{R}2)$ and $\widetilde K_R=(0,R)\times(0,R)$.
Let
$u\in H^1_0(\widetilde K_R)$ be a minimizer of $G_{b, \widetilde K_R}$. By magnetic translation invariance, we get
$m_0(b,R)=G_{b,\widetilde K_R}(u)$.
We extend $u$ to a function $\widetilde u\in H^1_0(\widetilde K_{nR})$ by `magnetic periodicity' as follows,
$$\widetilde u(x_1+R,x_2)=e^{i Rx_2/2}u(x_1,x_2)\,,\quad u(x_1,x_2+R)= e^{-i Rx_1/2}u(x_1,x_2)\,.$$
Using magnetic translation invariance,  it is easy to check that
$$G_{b,\widetilde K_{nR}}(\widetilde u)=n^2m_0(b,R)\,.$$
Consequently, we get $m_0(b,nR)\leq n^2m_0(b,R)$. We divide  both sides of this inequality by $n^2R^2$ then we take the limit as $n\to\infty$. That gives us $g(b)\leq m_0(b,R)$.

{\it Step~3.}

In this step we prove that the function $g$ is increasing. Let $b\geq 0$ and $\varepsilon>0$.
Since $\varepsilon>0$, it is easy to see that $m_0(b+\varepsilon,R)\geq m_0(b,R)$. Dividing both sides of this inequality by $R^2$ then taking $R\to\infty$ we get $g(b+\varepsilon)\geq g(b)$.

{\it Step~4.}

In this step we prove that $g$ is concave and continuous.
The concavity of $g$ is straight forward.
Upon writing
\begin{align}
m_0(b,R) = \inf_u \int \left(b|(\nabla-i\Ab_0)u|^2-|u|^2+\frac12|u|^4\right)\,dx,
\end{align}
we see that $b \mapsto m_0(b,R)$ is the infimum of a family of affine functions and therefore concave. So $g(b)$ is the pointwise limit of concave functions and therefore concave.

The concavity of $g$ implies continuity except at the endpoint $b=0$. Since $g(0)=-1/2$ and $g$ is non-decreasing it suffices to prove that $\displaystyle\limsup_{\varepsilon \to 0_{+}}g(\varepsilon)\leq -1/2$.
Consider $\varepsilon >0$. In Step~2, we proved that
\begin{equation}\label{eq-cont-g}
g(\varepsilon)\leq \frac{m_0(\varepsilon,R)}{R^2}\,,\end{equation}
for all $R\geq 1$. Let  $u\in H^1_0(K_R)$.
From \eqref{eq-cont-g} we get that
\begin{align}
\limsup_{\varepsilon \to 0_{+}}g(\varepsilon) &\leq R^{-2} \limsup_{\varepsilon \to 0_{+}} \int_{K_R}\left(\varepsilon|(\nabla-i\Ab_0)u|^2-|u|^2+\frac12|u|^4\right)\,dx  \nonumber\\
&= R^{-2} \int_{K_R}\left(-|u|^2+\frac12|u|^4\right)\,dx.
\end{align}
Since $u$ is arbitrary the result follows from Lemma~\ref{lem-b=0}.

{\it Step~5.} The assertion (4) in Theorem~\ref{thm-thmd-AS} follows
directly from Lemma~\ref{lem-rough'}.

This finishes the proof of Theorem~\ref{thm-thmd-AS}.
\end{proof}

\subsubsection{Periodic minimizers}\label{sec:periodic}\

Recall that for each $R>0$,  $K_{R}=(-R/2,R/2)\times(-R/2,R/2)$ is a square of side length $R$.
We introduce the following space,
\begin{multline}\label{eq-hc2-space1}
E_{R}=\bigg{\{}u\in H^1_{\rm loc}(\R^2;\C)~:~u(x_1+R,x_2)=e^{i
Rx_2/2 }u(x_1,x_2)\\
u(x_1,x_2+R )=e^{-iRx_1/2  }u(x_1,x_2)\,,~\forall~(x_1,x_2)\in\R^2\bigg{\}}\,.
\end{multline}

Notice that the periodicity conditions in
\eqref{eq-hc2-space1} are constructed in such a manner that, for any
function $u\in E_{R}$, the functions $|u|$, $|\nabla_{\Ab_0}u|$ and
the vector field
$\overline u \nabla_{\Ab_0}u$ are periodic with respect to
the given lattice.

Recall the functional $G_{b,\mathcal D}$ in \eqref{eq-LF-2D} above. We introduce the ground state energy,
\begin{equation}\label{eq-mp(b,R)}
m_{\rm p}(b,R)=\inf_{u\in E_R}
G_{b,K_R}(u)\,.
\end{equation}

Starting from a minimizing sequence, it is easy to see that when $b>0$, $G_{b,K_R}$ admit minimizers in the space $E_R$.
Various properties of the minimizers are collected in next proposition.

\begin{prop}\label{prop-periodic}
Let $b>0$ and let $u\in E_R$ be a minimizer of the functional $G_{b,K_R}$, i.e. $G_{b,K_R}(u)=m_{\rm p}(b,R)$.
Then $u$ has  the following properties.
\begin{enumerate}
\item $u$ is a solution of the equation,
\begin{equation}\label{eq-GL-R2}
-b(\nabla-i\Ab_0)^2u=(1-|u|^2)u\,,\quad{\rm in}~\R^2\,.\end{equation}
\item There exists a universal constant $C_{\max}>0$ such that
\begin{align}\label{eq:FHinfty}
\|u\|_{L^\infty(\R^2)}\leq \min\left(1,C_{\max}\left[\frac1b-1\right]_+^{1/2}\right)\,.
\end{align}
\end{enumerate}
\end{prop}
\begin{proof}
Since $u$ is a minimizer of $G_{b,K_R}$, then $u$ satisfies the equation in \eqref{eq-GL-R2} in $K_R$. Using the periodicity of
 $u$ together with the explicit definition of the magnetic potential $\Ab_0$, it is easy to check that $u$ satisfies the equation in $\R^2$.

Since $|u|$ is periodic, then $\|u\|_{L^\infty(\R^2)}=|u(x_0)|$ for some $x_0\in\overline{ K_R}$. It results from
a standard application of the strong maximum principle  that $|u(x_0)|\leq 1$. This proves that $\|u\|_{L^\infty(\R^2)}\leq 1$.

Since $u$ is a bounded solution of \eqref{eq-GL-R2}, then Property (2) of the proposition is a straight forward application of Theorem~3.1 in \cite{FH-jems}.
\end{proof}

In the next proposition, we exhibit the relation between the ground state energies  $m_0(b,R)$ and $m_{\rm p}(b,R)$.

\begin{prop}\label{prop-m0=mp}
Let  $m_0(b,R)$ and $m_{\rm p}(b,R)$ be as introduced in \eqref{eq-m0(b,R)} and \eqref{eq-mp(b,R)} respectively.
For all $b>0$ and $R>0$,  we have,
$$m_0(b,R)\geq m_{\rm p}(b,R)\,.$$
Furthermore, there exist universal constants $\epsilon_0\in(0,1)$ and $C>0$ such that,
if $b\geq 1-\epsilon_0$ and $R\geq 2$, then,
\begin{align}\label{eq:m0mP}
m_0(b,R)\leq m_{\rm p}(b,R)+C[1-b]_+R\,.
\end{align}
\end{prop}
\begin{proof}
Since every function $u\in H^1_0(K_R)$ can be extended by magnetic periodicity to a function in the space $E_R$, we get immediately
that  $m_0(b,R)\geq m_{\rm p}(b,R)$.

We prove the upper bound. Suppose that $R\geq 2$. Let $\chi_R\in C_c^\infty(\R^2)$ be a function satisfying,
$$0\leq\chi_R\leq 1\quad{\rm in~}\R^2\,,\quad {\rm supp}\,\chi_R\subset K_R\,,\quad
\chi_R=1\quad{\rm in}~K_{R-1}\,,\quad|\nabla\chi_R|\leq M\quad{\rm in~}\R^2\,,$$
   for some universal constant $M$.

Let $u\in E_R$ be a minimizer of $G_{b,K_R}$, i.e. $G_{b,K_R}(u)=m_{\rm p}(b,R)$. Let $f=\chi_R u$. Then $f\in H^1_0(K_R)$ and consequently
$$m_0(b,R)\leq G_{b,K_R}(f)\,.$$
The equation for $u$ and an integration by parts yield,
$$G_{b,K_R}(f)=\int_{K_R}b|\,|\nabla\chi_R|u\,|^2\,dx+\frac12\int_{K_R}|\chi_R u|^4\,dx
- \int \chi_R^2 |u|^4\,dx\,,$$
and
$$m_{\rm p}(b,R)=G_{b,K_R}(u)=-\frac12\int_{K_R}|u|^4\,dx\,.$$
Using the properties of $\chi_R$---and the bound $\| u \|_{\infty} \leq 1$---we deduce the following upper bound,
$$m_0(b,R)\leq m_{\rm p}(b,R)+4(M^2+1)R\|u\|^2_{L^\infty(K_R)}\,.$$
We get \eqref{eq:m0mP} by inserting \eqref{eq:FHinfty} in this
estimate.
\end{proof}

The next proposition gives a uniform upper bound of $m_{\rm p}(b,R)$ which is interesting when $b$ is close to $1$ and $R$ is large.
It is a key-ingredient in the proof of Theorem~\ref{thm-AS} below. Actually, this bound is also true with $m_{\rm p}(b,R)$ replaced by $m_{0}(b,R)$ but we will not need this version.

\begin{prop}\label{prop-proof-AS}
Let  $g(b)$ and $m_{\rm p}(b,R)$ be as defined in \eqref{eq-g(b)} and \eqref{eq-mp(b,R)}   respectively.
There exist universal constants $\epsilon_0\in(0,1)$, $R_0\geq 2$ and $C>0$ such that,
if $1-\epsilon_0\leq b<1$ and $R\geq R_0$, then,
$$
\frac{m_{\rm p}(b,R)}{R^2}\leq g(b)+\frac{C(1-b)^{3/2}}{R}\,.
$$
\end{prop}
\begin{proof}
Let $n\geq 2$ be a natural number. If $a\in(0,1)$ and $j=(j_1,j_2)\in\Z^2$, we denote by
$$K_{a,j}=I_{j_1}\times I_{j_2}\,,$$
where
$$\forall~m\in\Z\,,\quad I_m=\bigg{(}\frac{2m+1-n}2-\frac{(1+a)}2\,,\,\frac{2m+1-n}2+\frac{(1+a)}2\bigg{)}\,.$$
Consider a partition of unity $(\chi_j)$ of $\R^2$ such that:
$$\sum_{j}|\chi_j|^2=1\,,\quad 0\leq\chi_j\leq1\quad{\rm in}~\R^2\,,\quad {\rm supp}\chi_j\subset K_{a,j}\,,\quad
|\nabla\chi_j|\leq \frac{C}{a}\,,$$
where $C$ is a universal constant.
We define $\chi_{R,j}(x)=\chi_j(x/R)$. Then we obtain a new partition of unity $\chi_{R,j}$ such that
${\rm supp}\,\chi_{R,j}\subset \mathcal K_{R,j}$,
with
$$\mathcal  K_{R,j}=\{R x~:~x\in K_{a,j}\}\,.$$
Let $\mathcal J=\{j=(j_1,j_2)\in \Z^2~:~1\leq  j_1,j_2\leq n\}$ and $K_{nR}=(-nR,nR)\times(-nR,nR)$.
Then the family
$(\mathcal K_{R,j})_{j\in\mathcal J}$ is a covering of  $K_{nR}$, and is formed exactly of $n^2$ squares.

We  restrict the partition of unity $(\chi_{R,j})$ to the set  $K_{nR}=(-nR,nR)\times(-nR,nR)$.
Let $u_{nR}$ be a minimizer of $G_{b,K_{nR}}$ over the space $E_R$, i.e.
 $G_{b,K_{nR}}(u_{nR})=m_{\rm p}(b,nR)$.
We have the following decomposition formula (using the pointwise inequality $\sum \chi_{R,j}^4 \leq \sum \chi_{R,j}^2 = 1$),
\begin{equation}\label{eq-IMS-ONE}
G_{b,K_{nR}}(u_{nR})\geq \sum_{j\in\mathcal J}
\left(G_{b,K_{nR}}(\chi_{R,j}u_{nR})
-\|\,|\nabla\chi_{R,j}|\,u_{nR}\|^2_{L^2(K_{nR})}\right)\,.
\end{equation}
We reformulate \eqref{eq-IMS-ONE} as follows,
\begin{equation}\label{eq-IMS-PRIME}
G_{b,K_{nR}}(u_{nR})\geq \sum_{j\in\mathcal J}
\left(\frac{G_{b,K_{nR}}(\chi_{R,j}u_{nR})}{(1+a)^2}+\frac{a(2+a)}{(1+a)^2}G_{b,K_{nR}}(\chi_{R,j}u_{nR})
-\|\,|\nabla\chi_{R,j}|\,u_{nR}\|^2_{L^2(K_{nR})}\right)\,.
\end{equation}
Notice that each $\chi_{R,j}$ has support in a square of side length $(1+a)R$, hence it can be extended to a function in the space
$E_{(1+a)R}$ introduced in \eqref{eq-hc2-space1}. Therefore, using magnetic translation invariance, we get,
$$G_{b,K_{nR}}(\chi_{R,j}u_{nR})\geq m_{\rm p}(b,(1+a)R)\,.$$
We insert this lower bound into \eqref{eq-IMS-PRIME}.
To estimate the localization error we use that $|\nabla \chi_{R,j}| \leq C/(aR)$.
The support of $|\nabla \chi_{R,j}|$ is contained in an $(aR)$-neighborhood of a square of sidelength $(1+a)R$ and therefore, its
area is of the order $a R^2$.
So we get  from \eqref{eq-IMS-PRIME} the following lower bound,
\begin{equation}\label{eq-lb-IMS-'}
m_{\rm p}(b,nR)\geq n^2\frac{m_{\rm p}(b,(1+a)R)}{(1+a)^2}-n^2\frac{C}{a}\|u_{n\ell}\|^2_{L^\infty(K_{nR})}
+\frac{a(2+a)}{(1+a)^2}\sum_{j\in\mathcal J}G_{b,K_{nR}}(\chi_{R,j}u_{nR})\,.
\end{equation}
Here the constant $C$ is independent of $a$ and $R$. Since $\chi_{R,j}u_{nR}$ has compact support in $\R^2$, and the lowest eigenvalue of the Schr\"odinger operator with constant unit magnetic field in $L^2(\R^2)$ is equal to $1$,  we get by the variational min-max principle that,
$$G_{b,K_{nR}}(\chi_{R,j}u_{nR})\geq (b-1)\int_{K_{nR}}|\chi_{R,j}u_{nR}|^2\,dx\geq -(1-b)R^2\|u_{nR}\|_{L^\infty(K_{nR})}^2\,.$$
Inserting this into
 \eqref{eq-lb-IMS-'} then dividing both sides of the resulting inequality by $n^2R^2$, we get (using $a\leq 1$)
\begin{equation}\label{eq-lb1-IMS-'}
\frac{m_{\rm p}(b,nR)}{n^2R^2}
\geq \frac{m_{\rm p}(b,(1+a)R)}{(1+a)^2R^2}-\left(\frac{C}{aR^2}+3a(1-b)\right)\|u_{nR}\|^2_{L^\infty(K_{nR})}\,.
\end{equation}
We select $\epsilon_0$ such that $0<\epsilon_0<1/\sqrt{1+C_{\max}^2}$\,, where $C_{\max}$ is the universal  constant
from Proposition~\ref{prop-periodic}. In this way, we get
that $\|u_{nR}\|_{L^\infty(K_{nR})}\leq C_{\max}\sqrt{1-b}$ for all $b\in(1-\epsilon_0,1)$.

We infer from \eqref{eq-lb2-IMS-'},
\begin{equation}\label{eq-lb1-IMS-''}
\frac{m_{\rm p}(b,nR)}{n^2R^2}
\geq \frac{m_{\rm p}(b,(1+a)R)}{(1+a)^2R^2}-C\left(\frac{1}{a R^2}+a(1-b)\right) (1-b)\,,
\end{equation}
for all $b\in(1-\epsilon_0,1)$,
 $R\geq 2$, $n\in\N$ and $a\in(0,1)$.
Theorem~\ref{thm-thmd-AS} and Proposition~\ref{prop-m0=mp} together tell us that the function $m_p(b,\ell)/[\ell^2]$ has limit $g(b)$ as $\ell\to\infty$. Therefore, letting $n\to\infty$ on both sides of \eqref{eq-lb1-IMS-''} gives us,
\begin{equation}\label{eq-lb2-IMS-'}
g(b)\geq \frac{m_{\rm p}(b,(1+a)R)}{(1+a)^2R^2}-C\left(\frac{1}{a R^2}+a(1-b)\right) (1-b)\,.
\end{equation}
We select $a=\frac1{(1-b)^{1/2}R}$. It results from \eqref{eq-lb2-IMS-'} that
$$
g(b)\geq \frac{m_{\rm p}(b,\widetilde R)}{\widetilde R^2}-\frac{C(1-b)^{3/2}}{{\widetilde R}}
\,,$$
for all $b\in(\epsilon_0,1)$, $\widetilde R\geq 2/\sqrt{\epsilon_0}$ and some universal constant $C$. This finishes the proof of
Proposition~\ref{prop-proof-AS}.
\end{proof}

\subsubsection{The Abrikosov energy}\label{sec:Abrikosov}\

Recall the notation that if $R>0$,  $K_{R}=(-R/2,R/2)\times(-R/2,R/2)$ is a square of side length $R$.
In this section, we assume
the quantization condition that $|K_{R}|/(2\pi)$ is an integer, i.e.
there exists $N\in\mathbb N$ such that,
\begin{equation}\label{eq-hc2-quantization}
R^2=2\pi N\,.\end{equation}
Recall the definition of the space $E_R$ in \eqref{eq-hc2-space1}.
We denote by $P_{R}$ the operator,
\begin{equation}\label{eq-hc2-poperator}
P_{R}=-(\nabla-i\Ab_0)^2\quad{\rm in}~L^2(K_{R})\,,
\end{equation}
with form domain the space $E_{R}$ introduced in
(\ref{eq-hc2-space1}). More precisely, $P_{R}$ is the self-adjoint
realization associated with the closed quadratic form
\begin{equation}\label{eq-hc2-poperatorQF}
E_{R}\ni f\mapsto Q_{R}(f)=\|(\nabla-i\Ab_0)f\|_{L^2(K_{R})}^2\,.
\end{equation}

The operator $P_{R}$ being with compact resolvent, let us denote by
$\{\mu_j(P_{R})\}_{j\geq1}$ the increasing sequence of its distinct
eigenvalues (i.e. without counting multiplicity).

The following proposition may be classical in the spectral theory
of Schr\"odinger operators, but we refer to \cite{AS} or \cite{Al1}
for a simple proof.

\begin{prop}\label{prop-hc2-poperator}
Assuming $R$  is such that $|K_{R}|\in2\pi\mathbb N$, then the
operator $P_{R}$ enjoys the following spectral properties:
\begin{enumerate}
\item $\mu_1(P_{R})=1$ and $\mu_2(P_{R})\geq 3$\,.
\item The space $L_{R}={\rm Ker}(P_{R}-1)$ is finite
dimensional and ${\rm dim}\,L_{R}=|K_{R}|/(2\pi)$\,.
\end{enumerate}
Consequently, denoting by $\Pi_1$ the orthogonal projection on the
space $L_{R}$ (in $L^2(K_{R})$), and by $\Pi_2={\rm Id}-\Pi_1$, then
for all $f\in D(P_{R})$, we have,
$$\langle P_{R}\Pi_2 f\,,\,\Pi_2f\rangle_{L^2(K_{R})}\geq
3\|f\|^2_{L^2(K_{R})}\,.$$
\end{prop}

The next lemma is a consequence of the existence of a spectral gap
between the first two eigenvalues of $P_{R}$. It is proved in \cite[Lemma~2.8]{FK}.

\begin{lem}\label{lem-hc2-poperator'}
Given $p\geq 2$, there exists a constant $C_p>0$ such that, for any
$\gamma\in(0,\frac12)$, $R\geq 1$ with $|K_R| \in 2 \pi {\mathbb N}$, and $f\in D(P_{R})$ satisfying
\begin{equation}\label{eq-hc2-hypf}
Q_{R}(f)-(1+\gamma)\|f\|^2_{L^2(K_{R})}\leq0\,,\end{equation} the
following estimate holds,
\begin{equation}\label{eq-hc2-1=proj}
\|f-\Pi_1f\|_{L^p(K_{R})}\leq
C_p\sqrt{\gamma}\,\|f\|_{L^2(K_{R})}\,.
\end{equation}
Here $\Pi_1$ is the projection on the space $L_{R}$.
\end{lem}

We  introduce the following  energy functional (the Abrikosov
energy),
\begin{equation}\label{eq-hc2-eneAb}
F_{R}(v)=\int_{K_{R}}\left(-|v|^2+\frac12|v|^4\right)\,d
x\,.
\end{equation}
The energy $F_R$ will be minimized on the space
$L_{R}$, the eigenspace of the first eigenvalue of the periodic
operator $P_{R}$,
\begin{align}\label{eq-hc2-space2}
L_{R}&=\{u\in E_{R}~:~P_{R}u=u\}\nonumber\\
&=\{u\in
E_{R}~:~\left(\partial_{x_1}+i\partial_{x_2}+\frac12(x_1+ix_2)\right)u=0\}\,.
\end{align}
We set,
\begin{equation}\label{eq-hc2-c(r,t)}
c(R)=\min\{F_{R}(u)~:~u\in L_{R}\}\,,
\end{equation}

In the next theorem, we exhibit a connection between the ground state energies $m_{\rm p}(b,R)$ and $c(R)$. In the regime where the parameter $b$ is close to $1$, Theorem~\ref{thm-mp=c} is an improvement of Proposition~4.2 in \cite{AS}.

\begin{thm}\label{thm-mp=c}
Let $m_{\rm p}(b,R)$ and $c(R)$ be as   introduced in
\eqref{eq-mp(b,R)} and \eqref{eq-hc2-c(r,t)} respectively.
For all $b>0$ and $R>0$, we have,
$$m_{\rm p}(b,R)\leq [1-b]_+^2c(R)\,.$$
Furthermore, there exist universal constants $\epsilon_0\in(0,1)$ and $C>0$ such that, if $R\geq 2$,
$b\geq 1-\epsilon_0$,  and $0<\sigma<1/2$, then,
$$m_{\rm p}(b,R)\geq [1-b]_+^2\bigg((1+2\sigma)c(R)-C\sigma^{-3}(1-b)^2R^4\bigg)\,.$$
\end{thm}
\begin{proof}
We start by proving the upper bound on $m_{\rm p}(b,R)$. Let $v\in L_R$ be a minimizer of $F_R$, i.e.
$F_R(v)=c(R)$. Let $f=[1-b]_+^{1/2}v$. Then $f\in E_R$ and consequently,
$$m_{\rm p}(b,R)\leq G_{K_R}(f)\,.$$
Using that $v$ is an eigenfunction for the operator $P_R$, a simple computation gives,
\begin{align*}
G_{K_R}(f)&=[1-b]_+\int_{K_R}\left(-(1-b)|v|^2+\frac12[1-b]_+|v|^4\right)\,dx\\
&= [1-b]_+^2c(R)\,.\end{align*}

We prove the lower bound on $m_{\rm p}(b,R)$. Let $u\in E_R$ be a minimizer of $G_{K_R}$, i.e.
$G_{K_R}(u)=m_{\rm p}(b,R)$. Since
we do not know that $u\in L_R$, we approximate $u$ by its projection on the space $L_R$. Actually, we infer from
$m_{\rm p}(b,R)\leq 0$  the following inequality,
$$\int_{K_R} \left(|(\nabla-i\Ab_0)u|^2-|u|^2\right)\,dx\leq0\,.$$
Let $\gamma=[\frac1b-1]_+$. We infer from Lemma~\ref{lem-hc2-poperator'} that,
$$\|\Pi_1u-u\|_{L^4(K_R)}\leq C\sqrt{\gamma}\|u\|_{L^2(K_R)}\,,$$
where $\Pi_1$ is the projector on the space $L_R$.
We deduce from this inequality the following estimate,
$$\|u\|_{L^4(K_R)}\geq \|\Pi_1u\|_{L^4(K_R)}-C\sqrt{\gamma}\|u\|_{L^2(K_R)}\,.$$
This gives us, for some new constant $C$ and for all $\sigma\in(0,1)$,
\begin{equation}\label{eq-u>pi1}
\|u\|^4_{L^4(K_R)}\geq (1-\sigma)\|\Pi_1u\|^4_{L^4(K_R)}-C\sigma^{-3}\gamma^2\|u\|^4_{L^2(K_R)}\,.
\end{equation}
Now, we use the bound on $\|u\|_{L^\infty(\R^2)}$ given in Proposition~\ref{prop-periodic}. By selecting $\epsilon_0$ such that
$C_{\max}\sqrt{\epsilon_0}/\sqrt{1-\epsilon_0}<1$, we get for all $b\geq 1-\epsilon_0$,
\begin{equation}\label{eq-u>pi1'}
\|u\|^4_{L^4(K_R)}\geq (1-\sigma)\|\Pi_1u\|^4_{L^4(K_R)}-C\sigma^{-3}\gamma^4R^4\,.
\end{equation}
Using Proposition~\ref{prop-hc2-poperator} and variational min-max principle, we also get,
$$\int_{K_R}\left(b|(\nabla-i\Ab_0)u|^2-|u|^2\right)\,dx\geq (b-1)\int_{K_R}|\Pi_1u|^2\,dx\,.$$
This estimate, togther with that in \eqref{eq-u>pi1'} give us the following lower bound,
$$G_{K_R}(u)\geq \int_{K_R}\left(-(1-b)|\Pi_1u|^2+\frac14(1-\sigma)|\Pi_1u|^4\right)\,dx
-C\sigma^{-3}\gamma^4R^4\,.$$
Recall that $u$ is a minimizer of $G_{K_R}$, and that Proposition~\ref{prop-periodic} tells us that $u=0$ if $b\geq 1$. Therefore, the aforementioned lower bound is the same as,
\begin{equation}\label{eq-ab-lb}
m_{\rm p}(b,R)\geq \int_{K_R}\left(-[1-b]_+|\Pi_1u|^2+\frac14(1-\sigma)|\Pi_1u|^4\right)\,dx
-C\sigma^{-3}\gamma^4R^4\,.
\end{equation}
We introduce a function $v\in L_R$ such that,
$$\Pi_1u=\frac{[1-b]_+}{1-\sigma} v\,.$$
Notice that $v=0$ if $b\geq 1$, since $u=0$ in this case. Notice that $v$ is constructed so that the right hand side on \eqref{eq-ab-lb} becomes equal to,
$$\frac1{1-\sigma}F_R(v)\,.$$
Since $v\in L_R$, then $F_R(v)\geq c(R)$. Also, since $0<\sigma<1/2$, then $(1-\sigma)^{-1}\leq 1+2\sigma$.
Consequently, we get,
$$m_{\rm p}(b,R)\geq (1+2\sigma)  [1-b]_+^2c_0(R)
-C\sigma^{-3}\gamma^4R^4\,.$$
Recalling the definition of $\gamma=[\frac1b-1]_+$, the last estimate is nothing but the estimate of Theorem~\ref{thm-mp=c}.
\end{proof}

As consequence of Theorems~\ref{thm-thmd-AS} and \ref{thm-mp=c}, and Proposition~\ref{prop-proof-AS}, we
prove the following theorem.

\begin{thm}\label{thm-AS}
Let the functions $g(b)$ and $c(R)$ be as defined in \eqref{eq-g(b)} and \eqref{eq-hc2-c(r,t)} respectively.
The limits
$$
\lim_{b\to1_-}\frac{g(b)}{(1-b)^2}\quad {\rm and}\quad
\lim_{\substack{R\to\infty\\|K_{R}|/(2\pi)\in\mathbb
N}}\frac{c(R)}{|K_R|}
$$
exist and are equal to  a universal constant $E_2\in[-\frac12,0)$. That is,
\begin{equation}\label{eq-hc2-E2-TDL}
E_2=
\lim_{\substack{R\to\infty\\|K_{R}|/(2\pi)\in\mathbb
N}}\frac{c(R)}{|K_R|}=\lim_{b\to1_-}\frac{g(b)}{(1-b)^2}\,.
\end{equation}
\end{thm}

\begin{rem}
The result of Theorem~\ref{thm-AS} appears first in \cite[Theorems~1 \& 2]{AS}.
Our proof is based on a different approach from \cite{AS}. Also in \cite[p.~200]{AS},  the proof of the inequality,
$$\limsup_{\substack{R\to\infty\\|K_{R}|/(2\pi)\in\mathbb
N}}\frac{c(R)}{|K_R|}\leq\liminf_{b\to1_-}\frac{g(b)}{(1-b)^2}\,,$$
seems to rely on an estimate of the type we give in Proposition~\ref{prop-proof-AS}, which is missing in \cite{AS}.
\end{rem}

\begin{proof}[Proof of Theorem~\ref{thm-AS}]
Using Theorems~\ref{prop-m0=mp} and \ref{thm-mp=c},
we may write for all $b\in(\frac12,1)$ and $\sigma\in(0,\frac12)$,
\begin{equation}\label{eq-lim-E2-lb}
\frac{m_0(b,R)}{R^2}+\frac{C(1-b)}{R}
\leq (1-b)^2\frac{c(R)}{R^2}\end{equation}
and
\begin{equation}\label{eq-lim-E2-ub}
(1-b)^2\frac{c(R)}{R^2}\leq
\frac1{1+2\sigma}\left(\frac{m_{\rm p}(b,R)}{R^2}
+C\sigma^{-3}(1-b)^4R^2\right)
\,,\end{equation}
Here $m_0(b,R)$ is introduced in \eqref{eq-m0(b,R)}.
We get by dividing both sides in \eqref{eq-lim-E2-lb} by $(1-b)^2$ then taking $\displaystyle\liminf_{R\to\infty}$,
\begin{equation}\label{eq-lim-E2''}
\frac{g(b)}{(1-b)^2}\leq
\liminf_{R\to\infty}\frac{c(R)}{R^2}\,.
\end{equation}
Consequently, by taking $\displaystyle\limsup_{b\to1_-}$ on both sides of \eqref{eq-lim-E2''}, we get,
\begin{equation}\label{eq-liminf-E2}
\limsup_{b\to1_-}\frac{g(b)}{(1-b)^2}\leq
\liminf_{R\to\infty}\frac{c(R)}{R^2}\,.
\end{equation}
Using the upper bound for $m_{\rm p}(b,R)$ in Proposition~\ref{prop-proof-AS}, we infer from \eqref{eq-lim-E2-ub},
\begin{equation}\label{eq-lim-E2-ub'}
\frac{c(R)}{R^2}\leq
\frac1{1+2\sigma}\left(\frac{g(b)}{(1-b)^2}
+\frac{C}{(1-b)^{1/2}R}+C\sigma^{-3}(1-b)^2R^2\right)
\,.\end{equation}
Let $\varepsilon>0$ be given.
By definition of $\limsup$, there exists $\delta>0$ such that, if $1-\delta\leq b<1$, then
$\displaystyle\frac{g(b)}{(1-b)^2}\leq \limsup_{b\to 1_-} \frac{g(b)}{(1-b)^2}+\varepsilon$.
Inserting  this upper bound into \eqref{eq-lim-E2-ub'}, then selecting $b=1-R^{-5/4}$, we get for all $R\geq \delta^{-4/5}$,
$$\frac{c(R)}{R^2}\leq
\frac1{1+2\sigma}\left(\limsup_{b\to1_-}\frac{g(b)}{(1-b)^2}+\varepsilon+C R^{-3/8}
+C\sigma^{-3}R^{-1/2}\right)\,.$$
Taking successively  $\displaystyle\limsup_{R\to\infty}$,  then $\displaystyle\lim_{\varepsilon\to0_+}$ and $\displaystyle\lim_{\sigma\to0_+}$ on both sides of the above inequality, we get,
$$\limsup_{R\to\infty}\frac{c(R)}{R^2}\leq \limsup_{b\to1_-}\frac{g(b)}{(1-b)^2}\,.$$
Combining this inequality with that in \eqref{eq-liminf-E2}, we deduce that
\begin{equation}\label{eq-lim-E2-ub''}
E_2:=\lim_{R\to\infty}\frac{c(R)}{R^2}=\limsup_{b\to1_-}\frac{g(b)}{(1-b)^2}\,.
\end{equation}
We return to \eqref{eq-lim-E2-ub'} and select $R=(1-b)^{-3/4}$, so that $R\to\infty$ as $b\to1_-$. That way we get (for this choice of $R = R(b)$)
\begin{equation}\label{eq-lim-E2-ub'''}
\frac{c(R)}{R^2}\leq
\frac1{1+2\sigma}\left(\frac{g(b)}{(1-b)^2}
+C(1-b)^{1/4}+C\sigma^{-3}(1-b)^{1/4}\right)\,.
\end{equation}
Since $c(R)/R^2$ has a limit as $R\to\infty$ which is given in \eqref{eq-lim-E2-ub''}, we get by taking  $\displaystyle\liminf_{b\to1_-}$ on both sides of \eqref{eq-lim-E2-ub'''},
$$\limsup_{b\to1_-}\frac{g(b)}{(1-b)^2}= \lim_{R\to\infty}\frac{c(R)}{R^2}\leq \liminf_{b\to1_-}\frac{g(b)}{(1-b)^2}\,.$$

It results from the estimate \eqref{eq-prop-g(b)} in Theorem~\ref{thm-thmd-AS} that  $E_2\in[-\frac12,-\alpha]$. The constant $\alpha$ is universal and $\alpha\in(0,\frac12)$. This finishes the proof of the theorem.
 \end{proof}

\subsection{Three-dimensional limiting energy}\label{sec-LE3D}

Let us again consider a constant $b\geq0$. For any domain $\mathcal
D\subset \R^3$, we define the following Ginzburg-Landau energy,
\begin{equation}\label{eq-LF-3D}
F^{\rm 3D}_{b, \mathcal D}(u)=\int_{\mathcal D}\left(b|(\nabla-i\Fb)u|^2
-|u|^2+\frac1{2}|u|^4\right)\,dx\,.
\end{equation}
We will sometimes omit the parameter $b$ from the notation and write $F^{\rm 3D}_{\mathcal D}$
instead of $F^{\rm 3D}_{b, \mathcal D}$.

Here $\Fb$ is the canonical magnetic potential,
\begin{equation}\label{eq-hc2-Fb}
\Fb(x_1,x_2,x_3)=\frac12(-x_2,x_1,0)\,,\quad\forall~x=(x_1,x_2,x_3)\in
\R^3\,.\end{equation}
We introduce the ground state energy,
\begin{equation}\label{eq-m0-3D}
M_0(b,R)=\inf_{u\in H^1_0(Q_R;\C)} F^{\rm 3D}_{b,Q_R}(u)\,.
\end{equation}

Surprisingly, we find that the thermodynamic limit of the functional
$F^{\rm 3D}_{\mathcal D}$ is equal to the corresponding two-dimensional limit.

\begin{thm}\label{thm-3D-thmd}
Suppose $R>0$, $Q_R=(-R/2,R/2)\times K_R\subset \R^3$ and $K_R=(-R/2,R/2)\times(-R/2,R/2)\subset\R^2$. The following  statements hold true.
\begin{enumerate}
\item If $b\geq1$, then for all $R>0$,
 $M_0(b,R)=0$\,.
\item There exists a universal constant $M>0$ such that, for all $b\geq 0$ and $R>0$, we have,
$$R\,m_0(b,R)\leq M_0(b,R)\leq (R-2) m_0(b,R)+M\,,$$
where $m_0(b,R)$ is the ground state energy introduced in \eqref{eq-m0(b,R)}.
\item If $0<b<1$, then
$$\lim_{R\to\infty}\frac{M_0(b,R)}{|Q_R|}=g(b)\,,$$
where $g(b)$ is the constant from \eqref{eq-g(b)}.
\end{enumerate}
\end{thm}
\begin{proof}
We start by proving the statement corresponding to $b\geq1$. Using the configuration $u=0$ as a test function, we get obviously that
$$\inf_{u\in H^1_0(Q_R)} F^{\rm 3D}_{Q_R}(u)\leq 0\,.$$
Let $u\in H^1_0(Q_R)$. To get a lower bound,  it is sufficient to prove that
$F^{\rm 3D}_{Q_R}(u)\geq 0$.
We extend $u$ to a function $\widetilde u\in H^1(\R^3)$ by setting $\widetilde u=0$ in $\R^3\setminus Q_R$. Then,
\begin{equation}\label{eq-LE-1}
F^{\rm 3D}_{Q_R}(u)=\int_{\R^3}\left(b|(\nabla-i\Fb)\widetilde u|^2-|\widetilde u|^2+\frac12|\widetilde u|^4\right)\,dx\,.\end{equation}
Recall that  the bottom of the spectrum of the magnetic Schr\"odinger operator $(-\nabla-i\Fb)^2$ in $L^2(\R^3)$ is equal to $1$. Using the variational min-max principle, this gives the following lower bound,
$$\int_{\R^3}|(\nabla-i\Fb)\widetilde u|^2\,dx\geq \int_{\R^3}   |\widetilde u|^2\,dx\,.$$
Inserting this into \eqref{eq-LE-1}, we get that $F^{\rm 3D}_{Q_R}(u)\geq 0$.

We prove the statement corresponding to $0<b<1$. First we mention that if $x=(x_1,x_2,x_3)\in\R^3$, then we write
$x=(x_\bot,x_3)$ and $\nabla_{x_\bot}=(\partial_{x_1},\partial_{x_2})$.
Let $u\in H^1_0(Q_R)$. Notice that
\begin{align*}
|(\nabla-i\Fb)u|^2=|(\nabla_{x_\bot}-i\Ab_0)u|^2+|\partial_{x_3}u|^2
\geq|(\nabla_{x_\bot}-i\Ab_0)u|^2\,,
\end{align*}
This gives the following lower bound,
$$F^{\rm 3D}_{Q_R}(u)\geq \int_{-R/2}^{R/2} G_{K_R}\big(u(\cdot,x_3)\big)\,dx_3\,,$$
where $K_R=(-R/2,R/2)\times(-R/2,R/2)$ and $G_{K_R}$ is the functional introduced in \eqref{eq-LF-2D}. Since $u\in H^1_0(Q_R)$, then $u(\cdot,x_3)\in H^1_0(K_R)$, for $x_3\in(-R/2,R/2)$ almost everywhere. So we can write
$G_{K_R}\big(u(\cdot,x_3)\big)\geq m_0(b,R)$.

Consequently, we deduce that
$$M_0(b,R)\geq R\, m_0(b,R)\,.$$
To get a matching upper bound, we consider the test function
$f(x_1,x_2,x_3)=u_b(x_1,x_2) \chi_{R}(x_3)$, where $u_b(x_1,x_2)$ is
a minimizer of $G_{K_R}$ and the function $\chi_{R}(x_3)$ satisfies,
$$\chi_R\in C_c^\infty(\R)\,,\quad
0\leq \chi_R \leq 1, \quad
{\rm supp}\,\chi_R\in[-R/2,R/2]\,,\quad \chi_R=1{\rm ~in~}[-\frac{R}{2}+1,\frac{R}2-1]\,,$$
and $|\chi_R'|\leq M$ in $\R$, for some universal constant $M$.

We estimate the energy $F^{\rm 3D}_{Q_R}(f)$. We get, using the pointwise bound $\chi_R^4 \leq \chi_R^2$,
$$F^{\rm 3D}_{Q_R}(f)\leq \left(\int_{K_R}|u_b|^2\,dx_1dx_2\right)\int_{-R/2}^{R/2}|\chi_{R}'(x_3)|^2\,dx_3
+G_{K_R}(u_b)\int_{-R/2}^{R/2}|\chi_{R}(x_3)|^2\,dx_3\,.$$
By \eqref{eq:MaxPr} we have $|u_b|\leq 1$.
Also, using the properties of $\chi_R$ we get,
$$\int_{-R/2}^{R/2}|\chi_{R}'(x_3)|^2\,dx_3\leq 2M\,,\quad R-2\leq\int_{\R}|\chi_R(x_3)|^2\,dx_3\leq R\,.$$
 Thus, we get,
$$F^{\rm 3D}_{Q_R}(f)\leq 2M
+(R-2)G_{K_R}(u_b) \,.$$
Consequently, we obtain,
$$M_0(b,R)\leq 2M
+(R-2)m_0(b,R)\,.$$
This proves statement  (2) in Theorem~\ref{thm-3D-thmd}.

The last statement in Theorem~\ref{thm-3D-thmd} results straightforwardly from the inequality in the second statement.
Actually, we divide both sides of the inequlity by $R^3$ then we take
 $R\to\infty$.
\end{proof}

\section{A priori estimates}\label{sec-apest}
The aim of this section is to give {\it a priori} estimates on the solutions of the Ginzburg-Landau equations \eqref{eq-3D-GLeq}. Those estimates play an essential role in controlling the error resulting from various approximations.

The starting point is the following  $L^\infty$-bound  resulting from the maximum principle. Actually,
if $(\psi,\Ab)\in H^1(\Omega;\C)\times \dot H^1_{\Div,\Fb}(\R^3)$ is a solution of \eqref{eq-3D-GLeq}, then
\begin{equation}\label{eq-psi<1}
\|\psi\|_{L^\infty(\Omega)}\leq1\,.
\end{equation}

The set of estimates below is proved in \cite[Theorem~3.3 and Eq. (3.35)]{FH-p}
(see also \cite{Pa} for an earlier version).

\begin{thm}\label{thm-3D-apriori}~
\begin{enumerate}
\item Suppose $1\leq p\leq 6$. There exists a constant $C_p>0$ such that, if
$\kappa>0$, $H>0$ and $(\psi,\Ab)\in H^1(\Omega)\times \dot
H^1_{\Div,\Fb}(\R^3)$ is a solution of \eqref{eq-3D-GLeq}, then
\begin{equation}\label{eq-3D-ariori1}
\|\Ab-\Fb\|_{W^{2,p}(\Omega)}\leq C_p\frac{1+\kappa H+\kappa^2}{\kappa
  H}
\|\psi\|_{L^\infty(\Omega)}\|\psi\|_{L^2(\Omega)}\,.
\end{equation}
\item There exists a constant $C>0$ such that, if $\kappa>0$, $H>0$
  and $(\psi,\Ab)\in H^1(\Omega)\times \dot
H^1_{\Div,\Fb}(\R^3)$ is a solution of \eqref{eq-3D-GLeq}, then
\begin{equation}\label{eq-3D-apriori2}
\|\curl(\Ab-\Fb)\|_{L^2(\Omega)}
\leq\frac{C}{H}\|\psi\|_{L^\infty(\Omega)}\|\psi\|_{L^2(\Omega)}\,.\end{equation}
\end{enumerate}
\end{thm}

Theorem~\ref{thm-3D-apriori} is needed in order to obtain the improved {\it a priori} estimates of the next theorem. Similar estimates are given in \cite{Pa}.

\begin{thm}\label{thm-3D-ad-est}
Suppose that $0<\Lambda_{\min}\leq \Lambda_{\max}$. There exist constants $\kappa_0>1$ and
$C_1>0$ such that, if
$$\kappa\geq\kappa_0\,,\quad \Lambda_{\min}\leq\frac\kappa{H}\leq \Lambda_{\max}\,,$$
and $(\psi,\Ab)\in H^1(\Omega;\C)\times \dot H^1_{\Div,\Fb}(\R^3)$
is a solution of \eqref{eq-3D-GLeq}, then
\begin{align}
&\|(\nabla-i\kappa H\Ab)\psi\|_{C(\overline{\Omega})}
\leq C_1\sqrt{\kappa H}\|\psi\|_{L^\infty(\Omega)}\,,\label{eq-est1}\\
&
\|\Ab-\Fb\|_{W^{2,6}(\Omega)}\leq
C_1\left(\|\curl(\Ab-\Fb)\|_{L^2(\R^3)}+\frac1{\sqrt{\kappa H}}
\|\psi\|_{L^6(\Omega)}\|\psi\|_{L^\infty(\Omega)}
\right),\label{eq-est2'}\\
&\|\Ab-\Fb\|_{C^{1,1/2}(\overline{\Omega})}\leq C_1
\left(\|\curl(\Ab-\Fb)\|_{L^2(\R^3)}+\frac1{\sqrt{\kappa H}}\|\psi\|_{L^6(\Omega)}\|\psi\|_{L^\infty(\Omega)}\right)\,.
\label{eq-est2}
\end{align}
\end{thm}
\begin{proof}\

{\bf Proof of \eqref{eq-est1}:}

Suppose the estimate \eqref{eq-est1} were false. Then, there exists a sequence of points $(P_n)\subset\Omega$, two real sequences $(\kappa_n)$, $(H_n)$,
a sequence of solutions $(\psi_n,\Ab_n)$ of \eqref{eq-3D-GLeq} and a constant $\Lambda\in[\Lambda_{\min},\Lambda_{\max}]$ such that,
\begin{equation}\label{eq-proof-assumption}
\frac{|(\nabla-i\kappa_nH_n\Ab_n)\psi_n(P_n)|}{\sqrt{\kappa_n H_n}\,\|\psi_n\|_{L^\infty(\Omega)}}\to\infty\,,
\end{equation}
$$\kappa_n\to+\infty\,,\quad \frac{\kappa_n}{H_n}\to\Lambda\quad{\rm as}\quad n\to\infty\,.$$
Let $S_n=\|\psi_n\|_{L^\infty(\Omega)}$. Define the re-scaled functions,
$$\varphi_n(y)=S_n^{-1}\,e^{-i\sqrt{\kappa_nH_n}\,\Ab_n(P_n)\cdot y}\,\psi_n\left(P_n+\frac{y}{\sqrt{\kappa_nH_n}}\right)\,,$$
and
$$\Ab_n(y)=\frac{\Ab_n(P_n+y/\sqrt{\kappa_nH_n})-\Ab_n(y)}{1/\sqrt{\kappa_nH_n}}\,.$$
Two cases may occur:\\

{\it Case 1:}\hskip0.5cm
$\sqrt{\kappa_nH_n}\,{\rm dist}(P_n,\partial\Omega)$ is an unbounded sequence.\\

{\it Case 2:} \hskip0.5cm  $\sqrt{\kappa_nH_n}\,{\rm dist}(P_n,\partial\Omega)$ is a bounded sequence.\\

In both Cases~1 and 2, we apply a compactness argument by using elliptic estimates then a diagonal sequence argument to select a limiting function. Actually, we refer to \cite[Section~4.3]{FH-p} for the detailed proof in the two dimensional case, and to \cite[Lemma~12.5.4]{FH-b} for a precise statement in the three dimensional case.

Therefore, in Case~1,
we can extract a subsequence, still denoted $\varphi_n$, a function $\varphi\in C^1(\R^3)$ and a vector field $\widetilde \Fb\in C(\R^3)$ such that
$$\varphi_n\to\varphi\quad{\rm in}\quad C^1(K)\,,\quad \Ab_n\to\widetilde \Fb\quad {\rm in}\quad C(K)\,,$$
for any compact set $K\subset \R^3$. In particular, we get that,
$$
\frac{|(\nabla-i\kappa_nH_n\Ab_n)\psi_n(P_n)|}{\sqrt{\kappa_n H_n}\,\|\psi_n\|_{L^\infty(\Omega)}}\to
|(\nabla-i\widetilde\Fb)\varphi(0)|\,,$$
which is contradictory with \eqref{eq-proof-assumption}.

Similarly, in Case~2, we can extract a subsequence, still denoted $\varphi_n$, a function $\varphi\in C^1(\R^3_+)$ and a vector field $\widetilde \Ab\in C(\R^3_+)$ such that
$$\varphi_n\to\varphi\quad{\rm in}\quad C^1(K)\,,\quad \Ab_n\to\widetilde \Ab\quad {\rm in}\quad C(K)\,,$$
for any compact set $K\subset\overline{ \R^3_+}$\,. In particular, we get that,
$$
\frac{|(\nabla-i\kappa_nH_n\Ab_n)\psi_n(P_n)|}{\sqrt{\kappa_n H_n}\,\|\psi_n\|_{L^\infty(\Omega)}}\to
|(\nabla-i\widetilde\Ab)\varphi(0)|\,,$$
which is again contradictory with \eqref{eq-proof-assumption}.
Therefore, the estimate \eqref{eq-est1} should be true.\\

{\bf Proof of \eqref{eq-est2'}:}\\
Let $a=\Ab-\Fb$.
Since $\Div a=0$, we get
by regularity of the curl-div system see e.g. \cite{GT} (or \cite[Theorem~D.3.1]{FH-b} for a statement of the result),
\begin{equation}
\|a\|_{L^6(\R^3)}\leq C\|\curl a\|_{L^2(\R^3)}\,.\label{eq-est3}
\end{equation}
The second equation in \eqref{eq-3D-GLeq} reads as follows,
$$-\Delta a=\frac1{\kappa H}\IM(\overline\psi\,(\nabla-i\kappa H\Ab)\psi)\mathbf 1_\Omega\,.$$
By elliptic estimates (see e.g. \cite{GT} or \cite[Theorem~E.4.2]{FH-b}),
$$\|a\|_{W^{2,6}(\Omega)}\leq C(\|a\|_{L^6(\R^3)}+\|\Delta a\|_{L^6(\R^3)})\,.$$
The estimates in \eqref{eq-est1}  and \eqref{eq-est3} now give,
$$\|a\|_{W^{2,6}(\Omega)}\leq C\left(\|\curl a\|_{L^2(\R^3)}+\frac{C}{\sqrt{\kappa H}}\,\|\psi\|_{L^6(\Omega)}\|\psi\|_{L^{\infty}(\Omega)}\right)\,.$$

{\bf Proof of \eqref{eq-est2}:}\\
This is a consequence  of the Sobolev embedding of $W^{2,6}(\Omega)$ into $C^{1,1/2}(\Omega)$ and \eqref{eq-est2'}.
\end{proof}

The next result  is a rather weak  $L^\infty$-bound valid for all critical points of $\mathcal E^{\rm 3D}$ provided that the magnetic field strength $H$ is close to $H_{C_2}$.

\begin{thm}\label{thm-Linfty}
Suppose that $g_1:\R_+\to\R_+$ is a function such that
$$g_1(\kappa)\to\infty\,,\quad\frac{g_1(\kappa)}{\kappa}\to0\quad{\rm as}~\kappa\to\infty\,.$$
There exists a function $g_2:\R_+\to\R_+$ satisfying
$$g_2(\kappa)\to\infty\quad{\rm as}~\kappa\to\infty\,,$$
such that if $\kappa\geq 1$, $|H-\kappa|\leq g_1(\kappa)$ and $(\psi,\Ab)\in H^1(\Omega;\C)\times \dot H^1_{\Div,\Fb}(\R^3)$ is a critical point of the energy in \eqref{eq-3D-GLf}, then,
$$\|\psi\|_{L^\infty(\omega_\kappa)}\leq g_2(\kappa)\,.$$
Here
$$\omega_\kappa=\{x\in\Omega~:~{\rm dist}(x,\partial\Omega)\geq g_1(\kappa)/\kappa\}\,.$$
\end{thm}
\begin{proof}
The proof is very close to Theorem~2.1 in \cite{FH-jems}. We give the
details for the reader's convenience.

Suppose that the conclusion of Theorem~\ref{thm-Linfty} were
false. Then,   we may find  a number $N>1$ and sequences $\kappa_n$, $H_n$,
$(\psi_n,\Ab_n)$  such that,
$$\kappa_n\to+\infty\,,\quad \frac{\kappa_n}{H_n}\to 1\quad{\rm
  as}\quad
n\to+\infty\,,$$
\begin{equation}\label{eq-3D-psin}
N^{-1}
\leq \|\psi_n\|_{L^\infty(\omega_{\kappa_n})}\leq
1\,,\end{equation}
and $(\psi_n,\Ab_n)$ is a solution of \eqref{eq-3D-GLeq} for
$(\kappa,H)=(\kappa_n,H_n)$.

It results from Theorem~\ref{thm-3D-apriori} that
$\|\Ab_n-\Fb\|_{W^{2,6}(\Omega)}\leq C_1$ for some constant
$C_1>0$. Using the compactness of the embedding of $W^{1,6}(\Omega)$
in $C^{1,\alpha}(\Omega)$ for $0<\alpha<\frac12$, we may select a
susbsequence, still denoted $(\psi_n,\Ab_n)$, and a function
$\widetilde\Ab\in C^{1,\alpha}(\overline{\Omega})$ such that
\begin{equation}\label{eq-3D-An}
\Ab_n\to\widetilde \Ab\quad{\rm in}\quad C^{1,\alpha}(\overline{\Omega})\,.\end{equation}
It results from the estimate in \eqref{eq-3D-apriori2} and the assumption $H_n\to\infty$ that
\begin{equation}\label{eq-3D-Ab}
\curl\widetilde\Ab=1\quad{\rm in}~\overline{\Omega}\,.\end{equation}
Let $P_n\in\omega_{\kappa_n}$ be a point satisfying
$|\psi_n(P_n)|=\|\psi_n\|_{L^\infty(\omega_{\kappa_n})}$. Using
\eqref{eq-3D-psin}, we deduce that,
$$N^{-1}\leq|\psi(P_n)|\leq 1\,.$$
We may select a new subsequence, still denoted $(P_n;\psi_n,\Ab_n)$,
and a point $P\in\overline{\Omega}$
such that $P_n\to P$ as $n\to+\infty$.

We define the following re-scaled functions,
\begin{eqnarray*}
&&\ab_n(y)=
\frac{\Ab_n(P_n+y/\sqrt{\kappa_nH_n})-\Ab_n(P_n)}{1/\sqrt{\kappa_nH_n}}\\
&&\varphi_n(y)=e^{-i\sqrt{\kappa_nH_n}\Ab_n(P_n)\cdot
  y}\,\psi_n(P_n+y/\sqrt{\kappa_nH_n})\,.
\end{eqnarray*}
Since $g_1(\kappa_n)\to+\infty$ as $n\to+\infty$ by assumption, it
follows that for any $R>0$, we may select $n_0$ sufficiently large
such that $\ab_n$ and $\varphi_n$ are defined in $\{y\in\R^3~:~|y|\leq
R\}$ for all $n\geq n_0$.

Notice that $\ab_n$ is constructed so that $\Div\ab_n=0$. We then
infer from the equation of $\psi_n$ the following equation,
\begin{equation}\label{eq-3D-eq-n}
-\Delta\varphi_n-2i\ab_n\cdot\nabla\varphi_n+|\ab_n|^2\varphi_n=\frac{\kappa_n}{H_n}(1-|\varphi_n|^2)\varphi_n\,.\end{equation}

Consider $R>1$.
Using the definition of $\ab_n$ together with \eqref{eq-3D-An} and
\eqref{eq-3D-Ab} we deduce that,
$$\ab_n\to\widetilde\Fb(y)\quad{\rm in}\quad C^{0,\alpha}(B(0,R))\,,$$
where $\widetilde\Fb(y)=D\widetilde\Ab(y)$ satisfies
$\curl\widetilde\Fb=1$.

Since $|\varphi_n|\leq 1$ and $|\ab_n|\leq C_R$ in $B(0,2R)$, we get by
elliptic regularity that the sequence $(\varphi_n)$ is bounded in
$W^{2,6}(B(0,R))$. The compactness of the embedding
$W^{2,6}(B(0,R))\to C^{1,\alpha}(B(0,R))$, $\alpha<\frac12$,
 and a standard diagonal sequence argument give the existence of a function
 $\varphi\in C^{1,\alpha}_{\rm loc}(\R^3)$ such that
$$\varphi_n\to\varphi    \quad{\rm in}\quad C^{1,\alpha}(K)\,,$$
for each compact set $K\subset\R^3$. Furthermore, the function
$\varphi$ satisfies,
$$N^{-1}\leq\|\varphi\|_{L^\infty(\R^3)}\leq 1\,,$$
and
$$-(\nabla-i\widetilde\Fb)^2\varphi=(1-|\varphi|^2)\varphi\,.$$
Since $\curl\widetilde \Fb=1$, we may find a function $\phi$ such that
$\widetilde\Fb=\Fb+\nabla\phi$, where $\Fb=(-x_2/2,x_1/2,0)$.

Setting $u=e^{-i\phi}\varphi$, we get that $u\in L^\infty(\R^3)$ and satisfies the equation,
\begin{equation}\label{eq-limit}
-(\nabla-i\Fb)^2u=(1-|u|^2)u\,.\end{equation} At the same time $u$ satisfies,
\begin{equation}\label{eq-u}
\|u\|_{L^\infty(\R^3)}\geq N^{-1}\,.\end{equation}
But Proposition~12.5.1 in \cite{FH-b} tells us that the only bounded solution of \eqref{eq-limit} is $u=0$, thereby contradicting
\eqref{eq-u}. Therefore, Theorem~\ref{thm-Linfty} is true.
\end{proof}

\section{Energy estimates in small cubes}\label{sec-locest}

In this section,
the notation  $Q_\ell$ stands for  a cube in $\R^3$ of side length $\ell>0$ and whose axis is {\it parallel} to $\beta=(0,0,1)$,
i.e. $Q_\ell$ is of the form,
$$Q_\ell=(-\ell/2+a_1,a_1+\ell/2)\times (-\ell/2+a_2,a_2+\ell/2)\times(-\ell/2+a_3,a_3+\ell/2)\,,$$
where $a=(a_1,a_2,a_3)\in\R^3$.

If $(\psi,\Ab)\in H^1(\Omega;\C)\times \dot H^1_{\Div,\Fb}(\R^3)$, we denote by
$e(\psi,\Ab)=|(\nabla -\kappa H\Ab)\psi|^2-\kappa^2|\psi|^2+\frac{\kappa^2}2|\psi|^4$.
Furthermore, we define the Ginzburg-Landau energy of $(\psi,\Ab)$ in a domain $\mathcal D\subset\Omega$ as follows,
\begin{equation}\label{eq-GLen-D}
\mathcal E(\psi,\Ab;\mathcal D)=\int_{\mathcal D} e(\psi,\Ab)\,dx+(\kappa H)^2\int_{\R^3}|\curl(\Ab-\Fb)|^2\,dx\,.
\end{equation}
We also introduce the functional,
\begin{equation}\label{eq-GLe0}
\mathcal E_0(u,\Ab;\mathcal D)=\int_{\mathcal D}\left(|(\nabla-i\Ab)u|^2-\kappa^2|u|^2+\frac{\kappa^2}2|u|^4\right)\,dx\,.
\end{equation}
If $\mathcal D=\Omega$, we sometimes omit the dependence on the domain and write $\mathcal E_0(\psi,\Ab)$ for
$\mathcal E_0(\psi,\Ab;\Omega)$.

We start with a lemma that will
be useful in the proof of Proposition~\ref{prop-lb} below and also in Section~\ref{sec-hc2} of the paper.

\begin{lem}\label{lem-lb}
Let $0<\Lambda_{\rm min}< \Lambda_{\rm max}$. There exist positive constants $C$ and $\kappa_0$ such that if
\begin{align*}
\kappa_0 \leq \kappa,\qquad \Lambda_{\rm min}\leq \frac{H}{\kappa}\leq  \Lambda_{\rm max}\,,
\end{align*}
and
if $(\psi,\Ab)\in H^1(\Omega;\C)\times \dot H^1_{\Div,\Fb}(\R^3)$ is a critical point of \eqref{eq-3D-GLf},
 $f$ is a bounded function satisfying $\|f\|_{L^\infty(\Omega)}\leq 1$, ${\rm supp}\,f\subset \overline{Q_\ell}$, $f\psi\in H^1(Q_\ell)$
  and  $Q_\ell\subset\Omega$ is a cube of side length $\ell\in(0,1)$, then the  estimate,
$$\mathcal E_0(f\psi,\Ab; Q_{\ell})\geq (1-\delta)\mathcal E_0(e^{-i\phi_0}f\psi,\Fb;Q_{\ell})-C(\delta\kappa^2+\delta^{-1}\ell^2\lambda)
\int_{Q_\ell}|f\psi|^2\,dx\,,$$
holds true for all $\delta\in(0,1)$, $\kappa\geq\kappa_0$ and some
real-valued function $\phi_0\in H^1(\Omega)$. Here
$$\lambda=(\kappa H)^2\left(\|\curl(\Ab-\Fb)\|_{L^2(\R^3)}^2+\frac1{\kappa H}\|\psi\|^2_{L^6(\Omega)}\right)\,.$$
\end{lem}
\begin{proof}
Let $x_0$ be the center of the cube $Q_\ell$. Without loss of generality, we may assume that $x_0=0$. In this way, we reduce to the case where
$$Q_{\ell_\kappa}=(-\ell_\kappa/2,\ell_\kappa/2)\times (-\ell_\kappa/2,\ell_\kappa/2)\times (-\ell_\kappa/2,\ell_\kappa/2)\subset\Omega\,.$$
Let $\phi_0(x)=\big(\Ab(x_0)-\Fb(x_0)\big)\cdot x$, where $\Fb$ is the magnetic potential introduced in \eqref{eq-hc2-Fb}.
Invoking the estimate in \eqref{eq-est2}, we get,
$$|\Ab(x)-\nabla\phi_0-\Fb(x)|\leq C\frac{\sqrt{\lambda}}{\kappa H}\,\ell\,\qquad \forall~x\in Q_{\ell}.$$
Let $u=e^{-i\phi_0}f\psi$ and $\widetilde\Ab=\Ab-\nabla\phi_0$. Then
$\mathcal E_0(f\psi,\Ab;Q_{\ell})=\mathcal E_0(u,\widetilde \Ab;Q_\ell)$ and
\begin{equation}\label{eq-3D-est-g}
|\widetilde \Ab(x)-\Fb(x)|\leq
\frac{C\sqrt{\lambda}\,\ell}{\kappa H}\,,\quad\forall~x\in Q_{\ell}\,.
\end{equation}
We estimate the energy $\mathcal E_0(u,\widetilde \Ab;Q_{\ell})$ from below.
We start by estimating the kinetic energy from below as follows,
$$|(\nabla-i\kappa H\widetilde \Ab)u|^2\geq (1-\delta)|(\nabla-i\kappa H\Fb)u|^2+(1-\delta^{-1})(\kappa H)^2|(\widetilde\Ab-\Fb)u|^2\,,$$
for any $\delta\in(0,1)$.

Using the estimate in \eqref{eq-3D-est-g} together with the assumptions $H\approx \kappa$, $\|f\|_\infty\leq 1$
we deduce the  lower bound,
\begin{equation}
\mathcal E_0(f\psi,\Ab;Q_{\ell})\geq (1-\delta)\mathcal E_0(u,\Fb;Q_{\ell})-C(\delta\kappa^2+\delta^{-1}\ell^2\lambda)\int_{Q_\ell}|f\psi|^2\,dx\,.
\end{equation}
\end{proof}

\begin{prop}\label{prop-lb}
Let $0<\Lambda_{\rm min}< \Lambda_{\rm max}$. There exist positive constants $C, R_0$ and $\kappa_0$ such that the following is true.

Let $\ell>0$ satisfy $R_0\, \kappa^{-1} \leq \ell  \leq 1/2$ and let
\begin{align*}
\kappa_0 \leq \kappa,\qquad \Lambda_{\rm min}\leq \frac{H}{\kappa}\leq  \Lambda_{\rm max}\,.
\end{align*}
Then, if  $(\psi,\Ab)\in H^1(\Omega;\C)\times \dot H^1_{\Div,\Fb}(\R^3)$ is a critical point of \eqref{eq-3D-GLf}  and $Q_{\ell}\subset\Omega$ is a cube of side length $\ell$, then,
$$\frac1{|Q_{\ell}|}\mathcal E_0(\psi,\Ab;Q_{\ell})\geq g(H/\kappa)\kappa^2-C\left(\ell+\frac1{\ell \kappa}\right)\kappa^2\,,$$
for all $\kappa\geq 1$. Here $g(\cdot)$ is the function introduced in \eqref{eq-g(b)}.
\end{prop}
\begin{proof}
We may reduce to the case where
$$Q_{\ell}=(-\ell/2,\ell/2)\times (-\ell/2,\ell/2)\times (-\ell/2,\ell/2)\subset\Omega\,.$$
Recall the parameter $\lambda$ introduced in Lemma~\ref{lem-lb}.
Using the bound $\|\psi\|_\infty\leq 1$ and the estimate in
\eqref{eq-3D-apriori2}, we get $\lambda\leq C\kappa^2$. We apply
Lemma~\ref{lem-lb} with $f$  the characteristic function of
$Q_\ell$. After possibly performing a gauge transformation, we may
assume the following lower bound,
\begin{equation}\label{eq-3D-GLlb1}
\mathcal E_0(\psi,\Ab;Q_{\ell})\geq (1-\delta)\mathcal E_0(\psi,\Fb;Q_{\ell})-C(\delta+\delta^{-1}\ell^2)\kappa^2|Q_{\ell}|\,\|\psi\|_{L^\infty(Q_{\ell})}^2\,,
\end{equation}
where $C>0$ is a constant only depending on $\Omega$.

We estimate the energy $\mathcal E_0(\psi,\Fb;Q_{\ell})$ from below.  Let $b=H/\kappa$ and $R=\ell\sqrt{\kappa H}$\,.
Define the rescaled function,
$$\widetilde \psi(x)=\psi(x/\sqrt{\kappa H})\,,\quad\forall~x\in Q_R\,.$$
Recall the functional  $F^{\rm 3D}_{Q_R}$ introduced in \eqref{eq-LF-3D} above. It is easy to check that,
\begin{align}
\mathcal E_0(\psi,\Fb;Q_{\ell})&=\frac1{\sqrt{\kappa H}}\int_{Q_R}\left(|(\nabla-i\Fb)\widetilde\psi|^2-\frac{\kappa}{H}|\widetilde\psi|^2+\frac{\kappa}{2H}|\widetilde\psi|^4\right)\,dx\nonumber\\
&=\frac1{b\sqrt{\kappa H}}F^{\rm 3D}_{Q_R}(\widetilde\psi)\,.\label{eq-3D-GLlb2}
\end{align}
We still need to estimate from below the reduced energy $F^{\rm 3D}_{Q_R}(\widetilde\psi)$. The problem is that $\widetilde\psi$ is not in $H^1_0(Q_R)$. So, we introduce a cut-off function $\chi_R\in C_c^\infty(\R^3)$ such that
$$0\leq\chi_R\leq 1\quad{\rm in~}\R^3\,,\quad {\rm supp}\,\chi_R\subset Q_{R}\,,\quad  \chi_R=1\quad{\rm in~}
Q_{R-1}\,.$$
In particular, the function $\chi_R$ can be selected such that $|\nabla\chi_R|\leq C$ for some universal constant $C$.
Let $u=\chi_R\widetilde\psi$. It is easy to check that,
\begin{align*}
F^{\rm 3D}_{Q_R}(\widetilde\psi)&\geq \int_{Q_R}\left(b|\chi_R(\nabla-i\Fb)\widetilde\psi|^2- |\chi_R\widetilde\psi|^2+|\chi_R\widetilde\psi|^4\right)\,dx+\int_{Q_R}(\chi_R^2-1)|\widetilde\psi|^2\,dx\\
&\geq  F^{\rm 3D}_{Q_R}(u)+\int_{Q_R}(|\nabla\chi_R|^2+\chi_R^2-1)|\widetilde\psi|^2\,dx-2
|\langle (\nabla-i\Fb)(\chi_R\widetilde\psi),(\nabla\chi_R)\widetilde\psi\rangle|\,.
\end{align*}
The estimate in \eqref{eq-est1} tells us that  $|(\nabla-i\Fb)\widetilde\psi|$ is bounded independently of $\kappa$ and $H$.
Also, by using the bounds $|u|\leq1$, $|\nabla\chi_R|\leq C$ and the assumption on the support of $\chi_R$,
it is easy to check that,
\begin{align*}
F^{\rm 3D}_{Q_R}(\widetilde\psi)\geq  F^{\rm 3D}_{Q_R}(u)-CR^2\,\|\psi\|_{L^\infty(Q_{\ell})}^2\,.
\end{align*}
After recalling the definition of $M_0(b,R)$ introduced in \eqref{eq-m0-3D}, we get,
\begin{equation}\label{eq-3D-GLlb3}
F^{\rm 3D}_{Q_R}(\widetilde\psi)\geq M_0(b,R)-CR^2\,\|\psi\|_{L^\infty(Q_{\ell})}^2\,.
\end{equation}
We get by collecting the estimates in \eqref{eq-3D-GLlb1}-\eqref{eq-3D-GLlb3} that,
\begin{equation}\label{eq-3D-GLlb4}
\frac1{|Q_{\ell}|}\mathcal E_0(\psi,\Ab;Q_{\ell})\geq
(1-\delta)\frac1{b\ell^3\sqrt{\kappa H}} M_0(b,R)-r(\kappa)\,,\end{equation}
where
\begin{equation}\label{eq-3D-r1}
r(\kappa)=C\left((\delta+\delta^{-1}\ell^2)\kappa^2
+\frac{R^2}{b\ell^3\sqrt{\kappa H}}\right)\|\psi\|_{L^\infty(Q_{\ell})}^2\,.
\end{equation}
Theorems~\ref{thm-thmd-AS} and \ref{thm-3D-thmd} together tell us that $M_0(b,R)\geq R^3g(b)$ for all $b\in[0,1]$ and
$R$ sufficiently large. Here $g(b)$ is  introduced in \eqref{eq-g(b)}.  Therefore,  we get from \eqref{eq-3D-GLlb4} the estimate,
\begin{equation}\label{eq-3D-GLlb-b<1}
\frac1{|Q_{\ell}|}\mathcal E_0(\psi,\Ab;Q_{\ell})\geq  (1-\delta)\frac{R^3}{b\ell^3\sqrt{\kappa H}}
g(b)-
r(\kappa)\,.
\end{equation}
We choose $\delta=\ell$ so that
$$r(\kappa)=\mathcal O\big((\ell+(\ell\kappa)^{-1})\kappa^2\big).$$
After recalling that $R=\ell\sqrt{\kappa H}$ and  $b=H/\kappa$,
we finish  the proof of the proposition.
\end{proof}

\begin{prop}\label{prop-ub}
Let $0<\Lambda_{\rm min}< \Lambda_{\rm max}$. There exist positive constants $C, R_0$ and $\kappa_0$ such that the following is true.
Let $\ell>0$ satisfy $R_0 \kappa^{-1} \leq \ell  \leq 1/2$ and let
\begin{align*}
\kappa_0 \leq \kappa,\qquad \Lambda_{\rm min}\leq \frac{H}{\kappa}\leq  \Lambda_{\rm max}\,.
\end{align*}
Then,  if  $(\psi,\Ab)\in H^1(\Omega;\C)\times \dot H^1_{\Div,\Fb}(\R^3)$ is a critical point of \eqref{eq-3D-GLf},  and $Q_{\ell}\subset\Omega$ is a cube of side length $\ell$, then,
$$\frac1{|Q_{\ell}|}\mathcal E_0(\psi,\Ab;Q_{\ell})\leq g(H/\kappa)\kappa^2+C\left(\ell+\frac1{\ell\kappa}\right)\kappa^2\,.$$
Here $g(\cdot)$ is the function  introduced in \eqref{eq-g(b)}, and $\mathcal E_0$ is the functional in \eqref{eq-GLe0}.
\end{prop}
\begin{proof}
After performing a translation, we may assume that,
$$Q_{\ell}=(-\ell/2,\ell/2)\times (-\ell/2,\ell/2)\times (-\ell/2,\ell/2)\subset\Omega\,.$$
As explained earlier  in the proof of Lemma~\ref{lem-lb}, we may suppose, after performing a gauge transformation, that the magnetic potential $\Ab$ satisfies,
\begin{equation}\label{eq-3D-est-g'}
|\Ab(x)-\Fb(x)|\leq
\frac{C\ell}{\sqrt{\kappa H}}\,,\quad\forall~x\in Q_{\ell}\,,
\end{equation}
where $\Fb$ is the magnetic potential introduced in \eqref{eq-hc2-Fb}.

Let  $b=H/\kappa$, $R=\ell\sqrt{\kappa H}$ and $u_R\in H^1_0(Q_R)$ be a minimizer of the functional $F^{\rm 3D}_{Q_R}$
introduced in \eqref{eq-LF-3D}, i.e. $F^{\rm 3D}_{Q_R}(u_R)=M_0(b,R)$ where $M_0(b,R)$ is introduced in \eqref{eq-m0-3D}.

Let $\chi_R\in C_c^\infty(\R^3)$ be a cut-off function  such that,
$$0\leq\chi_R\leq 1\quad{\rm in~}\R^3\,,\quad {\rm supp}\,\chi_R\subset Q_{R+1}\,,\quad  \chi_R=1\quad{\rm in~}
Q_{R}\,.$$
and  $|\nabla\chi_R|\leq C$ for some universal constant $C$. Let $\eta_R(x)=1-\chi_R(x\sqrt{\kappa H})$ for all $x\in\R^3$.
We introduce the function,
\begin{equation}\label{eq-test-conf}
\varphi(x)=\mathbf 1_{Q_{\ell}}(x) u_R(x\sqrt{\kappa H})+\eta_R(x) \psi(x)\,,\quad \forall~x\in\Omega\,.
\end{equation}
Notice that by construction, $\varphi=\psi$ in $\Omega\setminus Q_{\ell+\frac1{\sqrt{\kappa H}}}$. We will prove that, for all $\delta\in(0,1)$,
\begin{equation}\label{eq-GLub}
\mathcal E(\varphi,\Ab;\Omega)\leq \mathcal E(\psi,\Ab;\Omega\setminus Q_{\ell})+(1+\delta)\frac1{b\sqrt{\kappa H}}
M_0(b,R)+ r_0(\kappa)\,,
\end{equation}
where $M_0(b,R)$ is defined by \eqref{eq-m0-3D}, and for some constant $C$, $r_0(\kappa)$ is given as follows,
\begin{equation}\label{eq-ub-r0}
r_0(\kappa)=C\ell^3\left(\delta+\delta^{-1}\ell^2+\frac1{\ell\kappa}\right)\kappa^2\,.
\end{equation}

Before proving \eqref{eq-GLub}, we explain how we get the result of the proposition. Actually, by definition of the minimizer $(\psi,\Ab)$, we have,
$$\mathcal E(\psi,\Ab)\leq \mathcal E(\varphi,\Ab;\Omega)\,.$$
Since $\mathcal E(\psi,\Ab;\Omega)=\mathcal E(\psi,\Ab;\Omega\setminus Q_{\ell})+\mathcal E_0(\psi,\Ab;Q_{\ell})$,
the estimate \eqref{eq-GLub} gives us,
$$\mathcal  E_0(\psi,\Ab;Q_{\ell})\leq (1+\delta)\frac1{b\sqrt{\kappa H}}
M_0(b,R)+ r_0(\kappa)\,.$$
Dividing both sides by $|Q_{\ell}|=\ell^3$ and remembering the definition of $r_0(\kappa)$, we get,
\begin{equation}\label{eq-ub-final}\frac1{|Q_{\ell}|}\mathcal E(\psi,\Ab,Q_{\ell})\leq
\frac{(1+\delta)}{b\ell^3\sqrt{\kappa H}}M_0(b,R)
+C \left(\delta+\delta^{-1}\ell^2+\frac1{\ell\kappa}\right)\kappa^2\,.\end{equation}
Theorems~\ref{thm-thmd-AS} and~\ref{thm-3D-thmd} together tell us that $M_0(b,R)\leq R^3g(b)+CR^2$ for all $b\in[0,1]$ and $R$ sufficiently large. We substitute this into \eqref{eq-ub-final} and we select $\delta=\ell$, so that $r_0(\kappa)=\kappa^2\mathcal O\big(\ell+(\ell\kappa)^{-1}\big)$.
Recalling that $R=\ell\sqrt{\kappa H}$, $b=H/\kappa$, we get,
$$\frac1{|Q_{\ell}|}\mathcal E(\psi,\Ab,Q_{\ell})\leq
g(b)\kappa^2+\kappa^2\mathcal O\big(\ell+(\ell\kappa)^{-1}\big)\,.$$
This establishes  the result of Proposition~\ref{prop-ub}.

{\bf Proof of \eqref{eq-GLub}:}

Recall the  Ginzburg-Landau energy $\mathcal E_0$  defined in \eqref{eq-GLe0}. We write,
\begin{equation}\label{eq-E=E1+E2}
\mathcal E_0(\varphi,\Ab;\Omega)=\mathcal E_1+\mathcal E_2\,,\end{equation}
where
\begin{equation}\label{eq-E1+E2}
\mathcal E_1=\mathcal E_0(\varphi,\Ab;\Omega\setminus Q_{\ell})\,,\quad \mathcal E_2=\mathcal E_0(\varphi,\Ab;Q_{\ell})\,.
\end{equation}
We estimate each of $\mathcal E_1$ and $\mathcal E_2$ from above. Starting with $\mathcal E_1$, we write,
\begin{equation}\label{eq-E1-ub}
\mathcal E_1=\mathcal E_0(\psi,\Ab;\Omega\setminus Q_{\ell})+\mathcal R(\psi,\Ab)\,,
\end{equation}
where
\begin{multline*}
\mathcal R(\psi,\Ab)=
\int_{\Omega\setminus Q_{\ell}}\bigg{(}(\eta_R^2-1)\left(|(\nabla-i\kappa H\Ab)\psi|^2-\kappa^2|\psi|^2\right)+|\psi\nabla\eta_R|^2+\frac{\kappa^2}2(\eta_R^4-1)|\psi|^4\\
+2\Re\langle  \eta_R(\nabla-i\kappa H\Ab)\psi,\psi\nabla\eta_R\rangle\bigg{)}\,dx\,.
\end{multline*}
Using that $\eta_R=1$ in $\Omega\setminus Q_{\ell+\frac1{\sqrt{\kappa H}}}$ together with the estimates $|\nabla\eta_R|\leq C\sqrt{\kappa H}$, $|\psi|\leq 1$ and
\eqref{eq-est1}, we get,
\begin{equation}\label{eq-est-R-ub}
|\mathcal R(\psi,\Ab)|\leq C\ell^2\kappa\,.
\end{equation}
Inserting \eqref{eq-est-R-ub} in \eqref{eq-E1-ub} we get the following estimate,
\begin{equation}\label{eq-E1-ub1}
\mathcal E_1\leq\mathcal E_0(\psi,\Ab;\Omega\setminus Q_{\ell})+C\ell^2\kappa\,.
\end{equation}
We estimate the term $\mathcal E_2$ from \eqref{eq-E1+E2}. We start by observing that for any $\delta\in(0,1)$, we have the following upper bound,
\begin{multline}\label{E2-ub}
\mathcal E_2\leq \int_{Q_{\ell}}\left\{(1+\delta)|(\nabla-i\kappa H\Fb)\varphi|^2-\kappa^2|\varphi|^2+\frac{\kappa^2}2|\varphi|^4\right\}\,dx\\
+(1+\delta^{-1})\kappa^2H^2\int_{Q_{\ell}}|\Ab-\Fb|^2|\varphi|^2\,dx\,.
\end{multline}
We use the estimates \eqref{eq-3D-est-g'} and $|\varphi|\leq 1$ and we get,
\begin{equation}
\label{E2-ub1}
\mathcal E_2\leq (1+\delta)
\int_{Q_{\ell}}\left(|(\nabla-i\kappa H\Fb)\varphi|^2-\kappa^2|\varphi|^2+\frac{\kappa^2}2|\varphi|^4\right)\,dx
+C\delta^{-1}\ell^{4}\kappa^2\,.
\end{equation}
Since $\varphi(x)=u_R(x\sqrt{\kappa H})$ in $Q_{\ell}$ and $R=\ell\sqrt{\kappa H}$, we get by performing a change of variables,
\begin{align}\label{E2-ub2}
&\hskip-2cm\int_{Q_{\ell}}\left(|(\nabla-i\kappa H\Fb)\varphi|^2-\kappa^2|\varphi|^2+\frac{\kappa^2}2|\varphi|^4\right)\,dx\nonumber\\
&\hskip1cm=
\frac1{\sqrt{\kappa H}}\int_{Q_R}\left(|(\nabla-i\Fb)u_R|^2-\frac{\kappa}{H}|u_R|^2+\frac{\kappa}{2H}|u_R|^4\right)\,dx\nonumber\\
&\hskip1cm=\frac1{b\sqrt{\kappa H}} F^{\rm 3D}_{Q_R}(u_R)\,,
\end{align}
where $F^{\rm 3D}_{Q_R}$ is the functional from \eqref{eq-LF-3D} (with $b=H/\kappa$).
Inserting \eqref{E2-ub2}  into \eqref{E2-ub1}, we get,
\begin{equation}\label{E1-ub3}
\mathcal E_2\leq (1+\delta)\frac1{b\sqrt{\kappa H}} F^{\rm 3D}_{Q_R}(u_R)+C\delta\ell^3\kappa^2+C\delta^{-1}\ell^{4}\kappa^2\,.
\end{equation}
Inserting \eqref{eq-E1-ub1} and \eqref{E1-ub3} into \eqref{eq-E=E1+E2}, we deduce that,
\begin{equation}\label{eq-ub-E1+E2}
\mathcal E_0(\psi,\Ab)\leq  \mathcal E_0(\psi,\Ab;\Omega\setminus Q_{\ell})+(1+\delta)\frac1{b\sqrt{\kappa H}} F^{\rm 3D}_{Q_R}(u_R)
+C\delta\ell^3\kappa^2+C\delta^{-1}\ell^{4}\kappa^2+C\ell^2\kappa\,.
\end{equation}
Recalling the definition of $\mathcal E(\psi,\Ab;\cdot)$ in \eqref{eq-GLen-D}, we see that  \eqref{eq-ub-E1+E2} is sufficient to finish the proof of \eqref{eq-GLub}.
\end{proof}

We conclude the section by giving an upper bound on $\E0(\kappa,H)$.

\begin{prop}\label{prop-ub-c0}
Let $0<\Lambda_{\rm min}< \Lambda_{\rm max}$ and let $\delta>0$. There exist positive constants $C$ and $\kappa_0$ such that if
\begin{align*}
\kappa_0 \leq \kappa,\qquad \Lambda_{\rm min}\leq \frac{H}{\kappa}\leq  \Lambda_{\rm max},
\end{align*}
then
the ground state energy $\E0(\kappa,H)$ in \eqref{eq-3D-gs} satisfies,
$$\E0(\kappa, H)\leq g(H/\kappa)|\Omega|\kappa^2+C\kappa^{1+\delta}\,.
$$
\end{prop}
\begin{proof}
Let $R=R(\kappa)$ be a positive parameter such that $1\ll R\ll\kappa$ as $\kappa\to\infty$. We will choose $R$ as a power of $\kappa$ at the end of the proof. Recall the magnetic potential $\Fb$ introduced in \eqref{eq-hc2-Fb} and the ground state energy $M_0(b,R)$
in \eqref{eq-m0-3D}.

Let $(\mathcal Q_j)$ be the lattice generated
by the cube $Q_{R/\sqrt{\kappa H}}$. Let $\mathcal I=\{j~:~\mathcal Q_j\subset\Omega\}$. We define $N={\rm Card}\,(\mathcal I)$. Then, as $R/\sqrt{\kappa H} \to 0$, $N$ satisfies,
$$N = \left(\frac{R}{\sqrt{\kappa H}}\right)^{-3}|\Omega|
+\mathcal O\left(\frac{\kappa H}{R^2}\right)\,.$$

Let $b=H/\kappa$, and $u$ a minimizer of the functional in \eqref{eq-LF-3D}, i.e.
$$M_0(b,R)=\int_{Q_R}\left(b|(\nabla-i\Fb)u|^2-|u|^2+\frac12|u|^4\right)\,dx\,.$$
Here recall the notation that if $r>0$, then  $Q_r=(-r/2,r/2)\times(-r/2,r/2)\times(-r/2,r/2)$.

Since $u\in H^1_0(Q_R)$, we extend  $u$ to all $\R^3$ by `magnetic periodicity' such that,
\begin{multline*}u(x_1+R,x_2,x_3)=e^{iRx_2/2}u(x_1,x_2,x_3),\quad
 u(x_1,x_2+R,x_3)=e^{-iRx_1/2}u(x_1,x_2,x_3),\\ u(x_1,x_2,x_3+R)=u(x_1,x_2,x_3)\,.\end{multline*}
For all $x\in\Omega$, define,
$$v(x)=\begin{cases}u(x\sqrt{\kappa H}), & x \in \cup_{j \in {\mathcal I}} {\mathcal Q}_j\\
0& \text{ if not}
\end{cases}
\,.$$
Since the original $u \in H^1_0(Q_R)$, this defines a $v \in H^1(\Omega)$.

We will compute the energy of the configuration $(v,\Fb)$. By periodicity, we get,
\begin{align}\label{eq-ub-C0}
\mathcal E(v,\Fb)=
N\times\int_{Q_{R/\sqrt{\kappa H}}}\left(|(\nabla-i\kappa H\Fb)v|^2-\kappa^2|v|^2+\frac{\kappa^2}2|v|^4\right)\,dx\,.
\end{align}

A change of variables gives us,
$$\int_{Q_{R/\sqrt{\kappa H}}}\left(|(\nabla-i\kappa H\Fb)v|^2-\kappa^2|v|^2+\frac{\kappa^2}2|v|^4\right)\,dx
=\frac{M_0(b,R)}{b\sqrt{\kappa H}}\,.$$

Inserting this in \eqref{eq-ub-C0}, we get,
$$\mathcal E(v,\Fb)\leq  N\times \frac{M_0(b,R)}{b\sqrt{\kappa H}}\,.$$
We know from Theorems~\ref{thm-thmd-AS} and \ref{thm-3D-thmd} that $M_0(b,R)\leq g(b) R^3+CR^2$ for all $b\in[0,1]$ and $R$ sufficiently large. Using this together with the estimate on the number $N$, we get,
$$\mathcal E(v,\Fb)\leq g(b) |\Omega|\kappa^2+C\frac{\kappa^2}{R}\,.
$$
We select $R=\kappa^{-1+\delta}$. Recalling that $b=H/\kappa$, this choice of $R$ and the aforementioned upper bound finishes the proof of the proposition.
\end{proof}

\section{Proof of main theorems}\label{sec-proofs}

\begin{proof}[Proof of Theorem~\ref{thm-3D-main}]\

In light of Proposition~\ref{prop-ub-c0}, we only need to establish a lower bound for the ground state energy $\E0(\kappa,H)$
in \eqref{eq-3D-gs}.

Let $\ell=\kappa^{-1/2}$. Then  $\kappa^{-1}\ll\ell_\kappa\ll1$ as $\kappa\to\infty$.
Consider a lattice $(Q_{j,\ell})_j$ of $\R^3$ generated by the cube,
$$Q_{\ell}=(-\ell/2, \ell/2)\times(-\ell/2, \ell/2) \times(-\ell/2, \ell/2)\,.$$
Let $\mathcal I=\{j~:~Q_{j,\ell}\subset\Omega\}$,  $N_\kappa={\rm Card}\,\mathcal I$ and
$\Omega_\kappa=\Omega\setminus\left(\displaystyle\cup_{j\not\in \mathcal I}Q_{j,\ell}\right)$. Notice that
$N_\kappa=|\Omega|\ell^{-3}+{\mathcal O}(\ell^{-2})$ as $\kappa\to\infty$.

If $(\psi,\Ab)$ is a minimizer of \eqref{eq-3D-GLf}, we have,
\begin{equation}\label{eq-glob-en}
\E0(\kappa,H)=\mathcal E_0(\psi,\Ab;\Omega_\kappa)+\mathcal E_0(\psi,\Ab;\Omega\setminus\Omega_\kappa) + (\kappa H)^2 \int_{{\mathbb R}^3} |\curl \Ab - \Fb|^2\,dx \,,
\end{equation}
where, for any $D\subset\Omega$, the energy $\mathcal E_0(\psi,\Ab;D)$ is introduced in \eqref{eq-GLe0}.
We may write,
\begin{equation}\label{eq-glob-en1}
\E0(\kappa,H)\geq\mathcal E_0(\psi,\Ab;\Omega_\kappa)+\mathcal E_0(\psi,\Ab;\Omega\setminus\Omega_\kappa)\,.
\end{equation}

Notice that $|\Omega\setminus\Omega_\kappa|\sim\ell|\partial\Omega|$ as $\kappa\to\infty$.
Thus, we get by using the estimate in \eqref{eq-est1},
\begin{equation}\label{eq-en-bnd}
\left|\mathcal E_0(\psi,\Ab;\Omega\setminus\Omega_\kappa)\right|\leq C\ell\kappa^2=C\kappa^{3/2}\,.
\end{equation}
To estimate $\mathcal E_0(\psi,\Ab;\Omega_\kappa)$, we notice that,
$$\mathcal E_0(\psi,\Ab;\Omega_\kappa)=\sum_{j\in\mathcal I}\mathcal E_0(\psi,\Ab;Q_{j,\ell})\,.
$$
Using Propositions~\ref{prop-lb} we get,
\begin{align}
\mathcal E_0(\psi,\Ab;\Omega_\kappa)\geq N_\kappa \left(g(b)\kappa^2-C\kappa^{3/2}\right)\ell^3 =|\Omega|g(b)\kappa^2-C\kappa^{3/2}\,.\label{eq-en-blk}
\end{align}
Inserting \eqref{eq-en-blk} and \eqref{eq-en-bnd} into \eqref{eq-glob-en} we deduce that,
$$\E0(\kappa,H)\geq g(b)|\Omega|\kappa^2-C\kappa^{3/2}\,,$$
thereby finishing the proof of Theorem~\ref{thm-3D-main}.
\end{proof}

\begin{proof}[Proof of Theorem~\ref{thm-3D-op}]\
It suffices to prove Theorem~\ref{thm-3D-op} in the  case where
$\ell \leq \kappa^{-1/2}$. In order to reduce to this case consider
a cube $Q_{\ell}$ with sidelength $\ell > \kappa^{-1/2}$. We can
subdivide this cube into $(\ell/\ell')^3$ identical cubes of
side-length $\ell' \in [\kappa^{-1/2}/2, \kappa^{-1/2})$. If the
result of Theorem~\ref{thm-3D-op} is valid for each of these smaller
cubes, the statement for $Q_{\ell}$ follows.

Therefore,
let $\ell$ be such that $\kappa^{-1}\ll \ell\leq \kappa^{-1/2}$, and $Q_{\ell}$
a cube of side-length $\ell$ with a side parallel to the external magnetic field $\beta$.
Let $(\psi,\Ab)$ be a solution of \eqref{eq-3D-GLeq}. Then $\psi$ satisfies,
\begin{equation}\label{eq-GLeq-1}
-(\nabla-i\kappa H\Ab)^2\psi=(1-|\psi|^2)\psi\,\quad{\rm in}\quad\Omega\,.
\end{equation}
We multiply both sides of the equation in \eqref{eq-GLeq-1} by $\overline{\psi}$ then we integrate over $Q_{\ell}$.
An integration by parts gives us,
$$\int_{Q_{\ell}}\left(|(\nabla-i\kappa H\Ab)\psi|^2-\kappa^2|\psi|^2+\kappa^2|\psi|^4\right)\,dx
+\int_{\partial Q_{\ell}}\nu\cdot(\nabla-i\kappa H\Ab)\psi\,\overline{\psi}\,d\sigma(x)=0\,.$$
Using the estimates \eqref{eq-psi<1} and \eqref{eq-est1}, we get that the boundary term above is $\mathcal O(\kappa\ell^2)$. So, we
rewrite the above equation as follows,
\begin{equation}\label{eq-proof-op}-\frac12\kappa^2\int_{Q_{\ell}}|\psi|^4\,dx=\mathcal E_0(\psi,\Ab;Q_{\ell})+\mathcal O(\kappa\ell^2)\,.\end{equation}
Using Proposition~\ref{prop-lb} and the assumption $\kappa^{-1}\ll\ell$, we conclude that
\begin{equation}\label{eq-op-ub}
\frac12\kappa^2\int_{Q_{\ell}}|\psi|^4\,dx\leq -g(H/\kappa)|Q_{\ell}|\kappa^2
+ C \left(\ell +\frac1{\ell\kappa} \right)\kappa^2\ell^3.
\end{equation}

If $(\psi,\Ab)$ is a minimizer of \eqref{eq-3D-GLf},  then \eqref{eq-proof-op} is still true. We apply in this case
Proposition~\ref{prop-ub} to write an upper bound for $\mathcal E_0(\psi,\Ab;Q_{\ell_\kappa})$. Consequently, we deduce that,
\begin{equation}\label{eq-op-lb}\frac12\kappa^2\int_{Q_{\ell}}|\psi|^4\,dx\geq -g(H/\kappa)|Q_{\ell}|\kappa^2-C \left(\ell +\frac1{\kappa\ell}\right)\kappa^2\ell^3\,.
\end{equation}
Combining the upper bound in \eqref{eq-op-lb} with the lower bound in \eqref{eq-op-ub} and using that by assumption $\ell \leq \kappa^{-1/2}$, finishes the proof of Theorem~\ref{thm-3D-op}.
\end{proof}

\section{Additional estimates}\label{sec-hc2}
The aim of this section is to give additional estimates on the energy which are  particularly interesting  in the case where the magnetic field satisfies
$$H=\kappa+o(\kappa)\quad{\rm as}\quad\kappa\to\infty\,.$$
These estimates will be used in \cite{FKP} to prove the asymptotic formula in \eqref{eq-FKP}.

\subsection{A lower bound}\

We will prove  Theorem~\ref{thm-hc2-lb} below, whose statement requires some notation.
Let $D\subset\Omega$ be a given  open set such that there exists a subset $\widetilde D$ of $\R^2$
having smooth boundary and $D=\widetilde D\cap\Omega$. For all $a>0$, we assign to $D$ the following subset of $\Omega$,
\begin{equation}\label{eq-Da}
D_a=\{x\in\Omega~:~{\rm dist}(x,D)\leq a\}\,.
\end{equation}
Let $g:\R_+\to\R_+$ be a function such that $g(\kappa)\to \infty$ as $\kappa\to\infty$. We put
$$\omega_\kappa=\{x\in\Omega~:~{\rm dist}(x,\partial\Omega)\geq g(\kappa)/\kappa\}\,.$$

\begin{thm}\label{thm-hc2-lb}
Suppose that the magnetic field $H$ is a function of $\kappa$ such that,
$$1\leq \liminf_{\kappa\to\infty}\frac{H}{\kappa}\leq\limsup_{\kappa\to\infty}\frac{H}{\kappa}<\infty\,.$$
Let $\kappa\ni\R_+\mapsto a(\kappa)\in\R_+$ be a function satisfying $\displaystyle\lim_{\kappa\to\infty}a(\kappa)=0$.
Then, for any solution   $(\psi,\Ab)\in H^1(\Omega;\C)\times \dot
H^1_{\Div,\Fb}(\R^3)$ of  \eqref{eq-3D-GLf} and any  function $h\in C^1(\Omega)$ satisfying $\|h\|_{L^\infty(\Omega)}\leq1$ and ${\rm supp}\,h\subset \overline{D_a}\cap\omega_\kappa$, the following asymptotic lower bound holds,
\begin{equation}\label{eq-3D-hc2-lb}
\mathcal E_0(h\psi,\Ab)\geq E_2|D|[\kappa-H]_+^2+o\bigg(\max\left(\kappa,[\kappa-H]_+^2\right)\bigg),\quad{\rm as}\quad\kappa\to\infty\,.
\end{equation}
Here $E_2<0$ is the universal constant introduced in \eqref{eq-hc2-E2-TDL}, and $\mathcal E_0$ is the functional
introduced in \eqref{eq-GLe0}.
\end{thm}

The proof of Theorem~\ref{thm-hc2-lb} will be split into several lemmas. One ingredient to control the errors resulting from various approximations is an $L^4$-bound on the order parameter given in \cite{Al} (also it is obtained in
  \cite[Chapter~12]{FH-b} by a different method).
We state this bound below.

\begin{lem}\label{lem-al}
Suppose the magnetic field $H$ is a function of $\kappa$ and satisfies,
$$1\leq\liminf_{\kappa\to\infty}\frac{H}{\kappa}\leq\limsup_{\kappa\to\infty}\frac{H}{\kappa}<\infty\,.$$
There exist positive constants $C$ and $\kappa_0$ such that, if $(\psi,\Ab)\in H^1(\Omega;\C)\times \dot H^1_{\Div,\Fb}(\R^3)$ is a solution of  \eqref{eq-3D-GLeq}, then,
$$\int_{\Omega}|\psi|^4\,dx\leq C\max\left(\frac1\kappa,\left[\frac{\kappa}H-1\right]^2_+\right)\,,$$
for all $\kappa\geq\kappa_0$.
\end{lem}
\begin{proof}
This is  a combination of the conclusions of Lemma~3.2 and Theorem~3.3 in \cite{Al}.
\end{proof}

The next lemma is taken from \cite[Lemma~10.33]{FH-b}, which, together with Lemma~\ref{lem-al}, give a good estimate of
$\|\curl(\Ab-\Fb)\|_{L^2(\R^3)}$.

\begin{lem}\label{lem-curl}
There exists a constant $C>0$ such that, if $(\psi,\Ab)\in H^1(\Omega;\C)\times \dot H^1_{\Div,\Fb}(\R^3)$ is a solution of  \eqref{eq-3D-GLeq}, then,
\begin{equation}\label{eq-est-curl}
\|\curl(\Ab-\Fb)\|_{L^2(\R^3)}\leq \frac{C}{H}\|\psi\|_{L^4(\Omega)}^2\,,
\end{equation}
for all $\kappa>0$ and $H>0$.
\end{lem}

\begin{lem}\label{lem-lb-hc2}
Suppose the assumptions in Theorem~\ref{thm-hc2-lb} hold.
There exist positive constants $C$ and $\kappa_0$ such that,
if $(\psi,\Ab)\in H^1(\Omega;\C)\times \dot H^1_{\Div,\Fb}(\R^3)$ is a critical point of \eqref{eq-3D-GLf},
 $f$ is a continuous function satisfying $\|f\|_{L^\infty(\Omega)}\leq 1$, ${\rm supp}\,f\subset Q_\ell$ and  $Q_\ell\subset\Omega$ is a cube of side length $\ell\in(0,1)$, then the  estimate,
\begin{multline}\label{eq-micro-lb}
\mathcal E_0(f \psi,\Ab)\geq (1-\delta)\frac{\ell}{\kappa H}[\kappa-H]_+^2
\left((1+2\sigma)c(\ell\sqrt{\kappa H})-C\sigma^{-3}(\kappa-H)^2H^2\ell^4\right)\\
-C(\delta\kappa^2+\delta^{-1}\ell^2\lambda)
\int_{Q_\ell}|f\psi|^2\,dx\,,\end{multline}
holds true for all $\delta\in(0,1)$, $\sigma\in(0,1/2)$ and
$\kappa\geq\kappa_0$. Here,
\begin{equation}\label{eq-lambda}
\lambda=(\kappa H)^2\left(\|\curl(\Ab-\Fb)\|_{L^2(\R^3)}^2+\frac1{\kappa H}\|\psi\|^2_{L^6(\Omega)}\right).
\end{equation}
\end{lem}
\begin{proof}
Using Lemma~\ref{lem-lb}, we get a real-valued function $\phi_0$ such that, with $u=e^{-i\phi_0}f\psi$, we may write,
\begin{equation}\label{eq-proof-m1}
\mathcal E_0(f \psi,\Ab)\geq (1-\delta) \mathcal E_0(u,\Fb)-C(\delta\kappa^2+\delta^{-1}\ell^2\lambda)
\int_{Q_\ell}|f\psi|^2\,dx\,.
\end{equation}
We estimate the energy $\mathcal E_0(u,\Fb)$ from below. Notice that ${\rm supp}\,u\subset Q_\ell$. By defining $R=\ell\sqrt{\kappa H}$ and the re-scaled function,
$$\forall~x\in Q_R\,,\quad v(x)=u(x/\sqrt{\kappa H})\,,$$
we get that $v\in H^1_0(Q_R)$. An easy computation gives,
$$
\mathcal E_0(u,\Fb)=\frac1{\sqrt{\kappa H}}\int_{Q_R}\left(|(\nabla-i\Fb)v|^2-\frac{\kappa}H|v|^2+\frac{\kappa}{2H}|v|^4\right)\,dx\,.
$$
Recall the definition of $M_0(b,R)$ in \eqref{eq-m0-3D}. With this notation, we get,
\begin{equation}\label{eq-proof-m2}
\mathcal E_0(u,\Fb)\geq \frac1{\sqrt{\kappa H}}\,\frac\kappa{H}\,M_0\left(\frac{H}\kappa,R\right)\,.\end{equation}
The conclusions in Theorem~\ref{thm-3D-thmd},  Theorem~\ref{thm-mp=c} and Proposition~\ref{prop-m0=mp} altogether
give,
$$ M_0\left(\frac{H}\kappa,R\right)\geq R \left[1-\frac{H}\kappa\right]_+^2
\left((1+2\sigma)c(R)-C\sigma^{-3}\left(1-\frac{H}\kappa\right)^2R^4\right)\,.$$
Recall that $R=\ell\sqrt{\kappa H}$. Substitution into \eqref{eq-proof-m2} gives,
\begin{equation}\label{eq-proof-m3}
\mathcal E_0(u,\Fb)\geq \frac{\ell}{\kappa H}[\kappa-H]_+^2
\left((1+2\sigma)c(\ell\sqrt{\kappa H})-C\sigma^{-3}(\kappa-H)^2H^2\ell^4\right)\,.
\end{equation}
Insersting \eqref{eq-proof-m3} into \eqref{eq-proof-m1} finishes the proof of Lemma~\ref{lem-lb-hc2}.
\end{proof}

\begin{proof}[Proof of Theorem~\ref{thm-hc2-lb}]
Consider two parameters  $\ell$ and $\alpha$ in the interval
$(0,1)$. The parameters $\ell$ and $\alpha$ will
be chosen as functions of $\kappa$  such that
$$\ell\to0\,,\quad \alpha\to0\,,\quad{\rm as}~
\kappa\to\infty\,.$$
We cover $\R^3$ by cubes $(Q_1(x_{j,\alpha}))_j$, where for all $j\in\mathbb Z^3$, $y=(y_1,y_2,y_3)\in\R^3$ and $\ell>0$, we define,
$$x_{j,\alpha}=(1-\alpha)j\,,\quad Q_\ell(y)=\prod_{k=1}^3\left(y_k-\frac\ell2,y_k+\frac\ell2\right)\,.$$
Let $(f_j)$ be a partition of unity in $\R^3$ such that,
$$\sum_j f_j^2=1\,,\quad {\rm supp}\,f_j\subset Q_1(x_{j,\alpha})\,,\quad |\nabla f_j|\leq \frac{C}{\alpha}\,,$$
for some universal constant $C$.

Defining the re-scaled functions,
$$\forall~x\in\R^3\,,\quad f_{j,\ell}(x)=f_j(x/\ell)\,,$$
we get a new partition of unity $(f_{j,\ell})_j$ such that each $f_{j,\ell}$ has support in a cube  of side length $\ell$ and,
$$\sum_j f_ {j,\ell}^2=1\,,\quad |\nabla f_{j,\ell}|\leq\frac{C}{\alpha\ell}\,.$$
Let $\mathcal J=\{j\in\mathbb Z^3~:~D_a\cap {\rm supp}\,f_{j,\ell}\not=\emptyset\}$ and $\mathcal N_\ell={\rm Card}\,\mathcal J$.
Then we know that
$$\mathcal N_\ell\times \ell^3\to |D_a|\quad{\rm as}~\ell\to0\,.$$
We have the localization formula,
\begin{equation}\label{eq-lb-mEN}
\mathcal E_0(h\psi,\Ab)=\sum_{j\in\mathcal J} \mathcal E_0(f_{j,\ell}h\psi,\Ab)-\sum_j\int_\Omega |\nabla f_{j,\ell}|^2|h\psi |^2\,dx\,.
\end{equation}
Using that $h\leq 1$,
\eqref{eq-lb-mEN}  gives us the following lower bound,
\begin{equation}\label{eq-lb-mET}
\mathcal E_0(h\psi,\Ab)\geq\sum_{j\in\mathcal J} \mathcal E_0(f_{j,\ell}h\psi,\Ab)-\frac{C}{\alpha^2\ell^2}\int_\Omega|\psi|^2\,dx\,.
\end{equation}
Many of the error terms will be controlled by the following parameter,
\begin{equation}\label{eq-lb-zeta}
\zeta=
\max\left(\kappa^{-1/2},\left[\frac{\kappa}H-1\right]_+\right)\,.
\end{equation}
Recall the parameter $\lambda$ from Lemma~\ref{lem-lb-hc2}. Using
Lemmas~\ref{lem-al} and \ref{lem-curl} we get the following  upper
bound on $\lambda$,
$$\lambda\leq C\kappa^2\zeta^{2/3}\,.$$
We apply Lemma~\ref{lem-lb-hc2} with $f=f_{j,\ell}$ to bound from
below each term $\mathcal E_0(f_{j,\ell}h\psi,\Ab)$.
This gives us,
\begin{multline}\label{eq-lb-m3}
\mathcal E_0(h\psi,\Ab)\geq \mathcal N_\ell
(1-\delta)\frac{\ell}{\kappa H}[\kappa-H]_+^2
\left((1+2\sigma)c(\ell\sqrt{\kappa H})-C\sigma^{-3}(\kappa-H)^2H^2\ell^4\right)\\
-C\bigg((\delta+\delta^{-1}\ell^2\zeta^{2/3})\kappa^2+\frac1{\alpha^2\ell^2}\bigg)
\int_{\Omega}|\psi|^2\,dx\,,
\end{multline}
for all $\delta\in(0,1)$ and $\sigma\in(0,1/2)$.

Theorem~\ref{thm-AS} tells us that if $(\ell\sqrt{\kappa H})^2$ is in $2\pi\mathbb N$ and large, then,
$$c(\ell\sqrt{\kappa H})=\ell^2\kappa H(E_2+o(1))\,.$$
Inserting this into \eqref{eq-lb-m3} and using that $\mathcal N_\ell\sim |D|\ell^{-3}$, we get,
\begin{multline}\label{eq-lb-m4}
\mathcal E_0(h\psi,\Ab)\geq \mathcal
(1-\delta)|D|
[\kappa-H]_+^2
\left((1+2\sigma)E_2-C\sigma^{-3}(\kappa-H)^2\ell^2\right)\\
-C\bigg((\delta+\delta^{-1}\ell\zeta^{2/3})\kappa^2+\frac1{\alpha^2\ell^2}\bigg)
\int_{\Omega}|\psi|^2\,dx\,.
\end{multline}
Also, using H\"older's inequality and
Lemma~\ref{lem-al} we get $\|\psi\|^2_{L^2(\Omega)}\leq C\zeta$. That
way, we infer from \eqref{eq-lb-m4} the bound,
\begin{multline}\label{eq-lb-mFOUR}
\mathcal E_0(h\psi,\Ab)\geq \mathcal
(1-\delta)|D|
[\kappa-H]_+^2
\left((1+2\sigma)E_2-C\sigma^{-3}\kappa^2\zeta^2\ell^2\right)\\
-C\bigg((\delta+\delta^{-1}\ell\zeta^{2/3})\kappa^2+\frac1{\alpha^2\ell^2}\bigg)\zeta\,.
\end{multline}
Let
$$\epsilon_1=\max(|\ln(\zeta^{11/6}\kappa)|^{-1},\zeta^{1/8})\,.$$ It
results from the definition of the parameter $\zeta$ that
$\epsilon_1\ll1$ as $\kappa\to\infty$. We select the parameters
$\delta$, $\alpha$ and $\ell$ as follows,
$$\alpha=\epsilon_1\,,\quad \ell=\frac1{\alpha \epsilon_1\kappa\sqrt{\zeta}}\,,\quad \delta=\epsilon_1\zeta\,.$$
It is easy to check that,
$$\kappa^2\zeta^2\ell^2\ll1\,,\quad
\frac{\zeta}{\alpha^2\ell^2}+\delta\kappa^2\zeta+\delta^{-1}\ell\zeta^{5/3}\kappa^2\ll\kappa^2\zeta^2\,,\quad
{\rm and}~\kappa^{-1}\ll\ell\,,$$
as $\kappa\to\infty$.

We select the parameter $\sigma$ as follows,
$$\sigma=\left(\kappa^2\zeta^2\ell^2\right)^{1/4}$$
so that $\sigma\ll1$ and $\sigma^{-3}\kappa^2\zeta^2\ell^2\ll1$ as
$\kappa\to\infty$. Thanks to this choice of parameters, and noticing
that $\kappa^2\zeta^2=\max(\kappa,[\kappa-H]_+^2)$, we infer from
\eqref{eq-lb-mFOUR},
$$  \mathcal E_0(h\psi,\Ab)\geq \mathcal
|D|
[\kappa-H]_+^2E_2+o\bigg(\max\left(\kappa,[\kappa-H]_+^2\right)\bigg)\,.$$
This finishes the proof of Theorem~\ref{thm-hc2-lb}.

\end{proof}

\subsection{Energy estimate of a trial configuration}\

The aim of this section is to prove Theorem~\ref{thm-tf-blk} below, which estimates the energy of a test configuration. The construction of the test configuration requires  some notation. Let $\chi$ be a cut-off function such that,
$$\chi\in C_c^\infty(\R)\,,\quad \supp\chi\subset[-2,2]\,,\quad 0\leq\chi\leq 1\quad{\rm in}~\R\,,\quad
\chi=1\quad{\rm in~} [-1,1]\,.$$
Let $\eta\in(0,1)$ and define a function $h_\eta\in C_c^\infty(\Omega)$ as follows,
$$\forall~x\in \Omega\,,\quad  h_\eta(x)=1-\chi\left(\frac{\dist(x,\partial\Omega)}{\eta}\right)\,.$$
Notice that $h_\eta$ vanishes in  a thin neighborhood of the boundary,
$$\Omega(\eta)=\{x\in\partial\Omega~:~\dist(x,\partial\Omega)\leq  \eta\}\,,$$
and $h_\eta=1$ in $\Omega\setminus\Omega(2\eta)$.

Consider a parameter $\ell\in(0,1)$. Let $b=H/\kappa$ and $R=\ell\sqrt{\kappa H}$.  Recall the ground state energy
$m_{\rm p}(b,R)$ defined in \eqref{eq-mp(b,R)}, together with the space $E_R$  in \eqref{eq-hc2-space1}.
Let $u_{b,R}\in E_R$ be a minimizer of the energy  in \eqref{eq-LF-2D}, i.e.
$$m_{\rm p}(b,R)=\int_{K_R}\left(b|(\nabla-i\Ab_0)u_{b,R}|^2-|u_{b,R}|^2+\frac12|u_{b,R}|^4\right)\,dx\,,
$$
where $K_R=(-R/2,R/2)\times(-R/2,R/2)$ and $\Ab_0(x_1,x_2)=(-x_2/2,x_1/2)$. By definition of the space $E_R$, the function $u_{b,R}\in H^1_{\rm loc}(\R^2)$.

We define the test function $\psi^{\rm blk}_{\eta,R}\in H^1(\Omega;\C)$ as follows,
\begin{equation}\label{eq-tf-blk}
\forall~x\in\Omega\,,\quad \psi^{\rm blk}_{\eta,R}(x)=h_\eta(x) \,u_{b,R}(\sqrt{\kappa H}\,x_\bot)\,.
\end{equation}
Here we used the notation that if $x=(x_1,x_2,x_3)\in\R^3$, then $x_\bot=(x_1,x_2)$.

\begin{thm}\label{thm-tf-blk}
Suppose that the magnetic field $H$ is a function of $\kappa$ such that,
$$1\leq \liminf_{\kappa\to\infty}\frac{H}{\kappa}\leq\limsup_{\kappa\to\infty}\frac{H}{\kappa}<\infty\,.$$
Suppose furthermore that $\ell$ and $\eta$ are  functions of $\kappa$ such that, as $\kappa\to\infty$,
$$\kappa^{-1}\ll\ell\ll1\,,\quad \kappa^{-1}\ll\eta\ll1\,.$$
Let $R=\ell\sqrt{\kappa H}$,  $\psi^{\rm blk}_{\eta,R}$ be the function introduced in \eqref{eq-tf-blk} and $\Fb(x)=(-x_2/2,x_1/2,0)$. Then, as $\kappa\to\infty$,
$$\mathcal E(\psi_{\eta,R}^{\rm blk},\Fb)\leq E_2|\Omega|\,[\kappa-H]_+^2+o\bigg(\max(\kappa,[\kappa-H]_+^2)\bigg)\,.$$
Here $E_2<0$
is the universal constant introduced in \eqref{eq-hc2-E2-TDL}, and $\mathcal E$ is the functional in \eqref{eq-3D-GLf}.
\end{thm}
\begin{proof}
Recall the magnetic potential $\Ab_0(x_1,x_2)=(-x_2/2,x_1/2)$ and that the function $u_{b,R}$ satisfies the equation\,,
$$-b(\nabla-i\Ab_0)u_{b,R}=(1-|u_{b,R}|^2)u_{b,R}\quad{\rm in~}\R^2\,.$$
By assumption on the magnetic field, we can make $[1-b]_+^{1/2}$
arbitrarily close to $0$ for all  $\kappa\geq\kappa_0$ by choosing
$\kappa_0>0$ sufficiently large. It results from
Proposition~\ref{prop-periodic} that $u_{b,R}$ satisfies the
universal bound $|u_{b,R}|\leq C_{\max}[1-b]_+^{1/2}$.

An integration by parts gives us,
\begin{equation}\label{eq-ub-hc2}
\int_{\R^3}\left(b|(\nabla-i\Fb)f u_{b,R}|^2-|fu_{b,R}|^2+f^2|u_{b,R}|^4\right)\,dx=\int_{\R^3}\big||\nabla f|u_{b,R}\big|^2\,dx\,,
\end{equation}
for any compactly supported smooth real-valued function $f\in C_c^2(\R^3)$ satisfying $0\leq f\leq 1$ in $\R^3$.
Consequently, we get
\begin{multline}\label{eq-ub1-hc2}
\int_{\R^3}\left(b|(\nabla-i\Fb)f u_{b,R}|^2-|fu_{b,R}|^2+\frac12|fu_{b,R}|^4\right)\,dx\leq \frac12\int_{\R^3}f^2(f^2-2)|u_{b,R}|^4\,dx
\\+C_{\max}^2[1-b]_+\int_{\R^3}|\nabla f|^2\,dx\,.
\end{multline}
Let $D=\supp f$ and $(Q_{R,j})_{j\in\Z^2}$ a uniform lattice of $\R^3$ whose unit cell is the cube
$$Q_R=K_R\times (-R/2,R/2)\,,\quad K_R=(-R/2,R/2)\times(-R/2,R/2)\,.$$
Let $\mathcal J=\{j\in\Z^2~:~Q_{R,j}\cap D\not=\emptyset\}$ and $N={\rm card}\,\mathcal J$. Recall that $|u_{b,R}|$ is periodic in $\R^2$. Then, writing
\begin{align*}
\int_{\R^3}f^2(2-f^2)|u_{b,R}|^4\,dx&=
\int_{\cup_{j\in\mathcal J}Q_{R,j}}|u_{b,R}|^4\,dx+\int_{\cup_{j\in\mathcal J}Q_{R,j}}(2f^4-f^2-1)|u_{b,R}|^4\,dx\\
&=NR\int_{K_R}|u_{b,R}|^4\,dx_\bot+\int_{\cup_{j\in\mathcal J}Q_{R,j}}(2f^4-f^2-1)|u_{b,R}|^4\,dx\,.
\end{align*}
We use the bounds $|u_{b,R}|\leq C_{\max} [1-b]_+^{1/2}$ and $0\leq f\leq 1$, and we notice that
an integration by parts gives $m_{\rm p}(b,R)=-\frac12\int_{K_R}|u_{b,R}|^4 \,dx$. Consequently, we obtain that,
$$\int_{\R^3}f^2(2-f^2)|u_{b,R}|^4\,dx\geq -2NR\,m_{\rm p}(b,R)+3C_{\max}^4
[1-b]_+^2\int_{\cup_{j\in\mathcal J}Q_{R,j}}(f^2-1)\,dx\,.$$
Inserting this into \eqref{eq-ub1-hc2} we get,
\begin{multline}\label{eq-ub2-hc2}
\int_{\R^3}\left(b|(\nabla-i\Fb)f u_{b,R}|^2-|fu_{b,R}|^2+\frac12|fu_{b,R}|^4\right)\,dx\\
\leq
NR\,m_{\rm p}(b,R)+C_{\max}^2[1-b]_+\int_{\cup_{j\in\mathcal J}Q_{R,j}}\left(|\nabla f|^2+\frac{3C_{\max}[1-b]_+}2(1-f^2)\right)\,dx\,.
\end{multline}
We choose $f(x)=h_\eta(x/\sqrt{\kappa H})$. Performing a change of variable in \eqref{eq-ub2-hc2}, we get (recall that $b=H/\kappa$ by definition),
\begin{multline}\label{eq-ub3-hc2}
\frac{H\sqrt{\kappa H}}{\kappa}
\int_{\R^3}\left(|(\nabla-i\kappa H\Fb)\psi^{\rm blk}_{\eta,R}|^2-\kappa^2|\psi^{\rm blk}_{\eta,R}|^2+\frac{\kappa^2}
2|\psi^{\rm blk}_{\eta,R}|^4\right)\,dx\\
\leq
NR\,m_{\rm p}(b,R)+MC_{\max}^2[1-b]_+\left(\frac{1}{\eta}
+\frac{\eta \kappa HC_{\max}[1-b]_+}{2}\right)\sqrt{\kappa H}\,,
\end{multline}
where $M$ is a universal constant. With this choice of $f$, it is easy to check that the number $N$ satisfies,
$$N\times \left(\frac{R}{\sqrt{\kappa H}}\right)^3\to|\Omega|\,,\quad {\rm as}\quad R\to\infty\,.$$
We insert this into \eqref{eq-ub3-hc2}. Theorems~\ref{thm-mp=c} and \ref{thm-AS}
together tell us that
$$m_{\rm p}(b,R)\leq R^2 [1-b]_+^2(E_2+o(1))\quad{\rm as}\quad R\to\infty\,,$$
where $E_2$ is a universal constant. Therefore, we deduce from \eqref{eq-ub3-hc2} that,
$$\mathcal E(\psi^{\rm blk}_{\eta,R},\Fb)\leq
\kappa^2[1-b]_+^2E_2(1+o(1))+M'C_{\max}^2[1-b]_+\left(\frac{1}{\eta}+\frac{\eta C_{\max}\kappa^2[1-b]_+}{2}\right)\,.
$$
where $M'$ is a new universal constant.

Since $\kappa^{-1}\ll\eta\ll 1$ as $\kappa\to\infty$, it is easy to check that
$$\frac{[1-b]_+}{\eta}\ll \kappa^2 \max\left(\frac1\kappa,[1-b]_+^2\right)\,.$$ Since $b=H/\kappa$,
this finishes the proof of the theorem.
\end{proof}

\section*{Acknowledgements}
The authors were partially supported by the Lundbeck foundation. SF was partially supported by the European Research Council under the European Community's
Seventh Framework Programme (FP7/2007-2013)/ERC grant agreement
n$^{\rm o}$ 202859.  AK was partially  supported by a grant from  the Lebanese University.

\end{document}